%% file: main-LICS.tex
\documentclass[a4paper,UKenglish,numberwithinsect]{lipics-v2021}
\pdfoutput=1
\hideLIPIcs
\title{The Finite Length Property of~the~Rado~Graph~and~Friends}

\author{Jingjie Yang}{University of Oxford, UK}
{jingjieyang2001@gmail.com}
{https://orcid.org/0009-0002-5393-6083}
{}

\author{Mikołaj Bojańczyk}{University of Warsaw, Poland}
{bojan@mimuw.edu.pl}
{https://orcid.org/0000-0002-7758-1072}
{}

\author{Bartek Klin}{University of Oxford, UK}
{bartek.klin@cs.ox.ac.uk}
{https://orcid.org/0000-0001-5793-7425}
{}

\funding{Mikołaj Bojańczyk was supported by the Polish National Science Centre (NCN) grant ``Polynomial finite state computation'' (2022/46/A/ST6/00072).
Jingjie Yang was partially supported by the Polish National Science Centre (NCN) grant ``Linear algebra in orbit-finite dimension'' (2022/45/N/ST6/03242).}

\authorrunning{J. Yang, M. Bojańczyk, and B. Klin} 

\Copyright{Jingjie Yang, Mikołaj Bojańczyk, and Bartek Klin} 

\ccsdesc[500]{Theory of computation~Logic}
\ccsdesc[500]{Mathematics of computing~Random graphs}
\ccsdesc[500]{Theory of computation~Automata over infinite objects}

\keywords{Rado graph, oligomorphic structure, orbit-finite set, orbit-finitely spanned vector space, equivariant subspace, finite length}

\relatedversiondetails[]{Conference Version}{https://doi.org/10.4230/LIPIcs.LICS.2026.82} 

\acknowledgements{We thank David M. Evans, Arka Ghosh, and Ehud Hrushovski for helpful discussions. 
We are very grateful to the anonymous reviewers for their insightful comments.}

\nolinenumbers

\EventEditors{Claudia Faggian and Joost-Pieter Katoen}
\EventNoEds{2}
\EventLongTitle{41st Annual Symposium on Logic in Computer Science (LICS 2026)}
\EventShortTitle{LICS 2026}
\EventAcronym{LICS}
\EventYear{2026}
\EventDate{July 20--23, 2026}
\EventLocation{Lisbon, Portugal}
\EventLogo{}
\SeriesVolume{380}
\ArticleNo{82}

\usepackage{mathtools}

\usepackage{cancel}

\usepackage{quiver}
\usepackage{tikz}
\usetikzlibrary{automata,calc,backgrounds,arrows.meta,decorations.pathmorphing,positioning,automata}
\usepackage[all]{xy}

\usepackage{booktabs}

\newtheorem{fact}[theorem]{Fact}

\theoremstyle{definition}
\newtheorem{non-example}[theorem]{Non-example}

\newcommand{\Bclass}{\mathscr{B}}
\newcommand{\Cclass}{\mathscr{C}}
\newcommand{\Fclass}{\mathscr{F}}
\newcommand{\Jclass}{\mathscr{J}}

\newcommand{\A}{\mathbb{A}}
\newcommand{\B}{\mathbb{B}}
\newcommand{\N}{\mathbb{N}} 
\newcommand{\Q}{\mathbb{Q}} 
\newcommand{\W}{\mathbb{W}} 

\newcommand{\FF}{\mathsf{F}} 
\newcommand{\ff}{\innerfield} 
\newcommand{\EE}{\mathsf{E}} 

\newcommand{\Aut}{\operatorname{Aut}}
\newcommand{\Lin}{\operatorname{Lin}}

\newcommand{\Ker}{\operatorname{BLin}}
\newcommand{\Cog}{\operatorname{Cog}}
\newcommand{\Forb}{\operatorname{Forb}}

\newcommand{\vsup}[1]{[#1]}

\newcommand{\field}{\FF}
\newcommand{\innerfield}{\mathsf{k}}

\newcommand{\set}[1]{\{#1\}}

\newcommand{\setbuild}[2]{\set{#1 \mid \text{#2}}}

\newcommand{\myunderbrace}[2]{\underbrace{#1}_{\text{
\begin{tabular}{c}
	#2
\end{tabular} }}}

\newcommand{\lineqfun}[2]{(#1 \underset{\textup{lineq}}{\longrightarrow} #2)}
\newcommand{\linfsfun}[2]{(#1 \underset{\textup{linfs}}{\longrightarrow} #2)}

\newcommand{\fsfun}[2]{(#1 \underset{\textup{fs}}{\longrightarrow} #2)}

\newcommand{\bigdot}{\boldsymbol{\cdot}}

\begin{document}

\maketitle

\begin{abstract}
  An infinite structure has the finite length property (over a given field) if, for each of its finite powers, chains of equivariant subspaces in the corresponding free vector space are bounded in length. 
  Prior work showed that the countable pure set and the countable dense linear order without endpoints have this property. 
  We generalise these results to 
  (a)~any structure approximated by finite substructures with few orbits, provided the field is of characteristic zero, and 
  (b)~any Fraïssé limit with free amalgamation in a finite vocabulary consisting of unary and binary relations, possibly expanded with a generic total order. 
  As a special case, we deduce the finite length property of the Rado graph using both methods. 
  We also describe some connections with function spaces, weighted register automata, and orbit-finite systems of linear equations.
\end{abstract}

\input{src/Intro}
\input{src/Structures}
\input{src/Sets}
\input{src/Spaces}
\input{src/Rado-cogless}
\input{src/Duals}

\input{src/Rado-cogs}
\input{src/Conclusion}

\emergencystretch=1em
\bibliography{atoms.bib}

\newpage
\appendix
\input{src/appendix-acp}

\input{src/appendix-symplectic}
\input{src/appendix-function-spaces}
\input{src/appendix-woodrow-lachlan}

\input{src/appendix-cogs.tex}
\input{src/appendix-coefficients.tex}

\end{document}

%% file: src/Intro.tex
\section{Introduction}
This paper is part of a research programme focused on orbit-finite sets and structures. 
In this programme, one starts with an infinite relational structure $\A$ whose elements are called \emph{atoms}. 
Based on these, one constructs sets that are called \emph{orbit-finite}. 
Precise definitions will follow, but the understanding is that elements of an orbit-finite set are constructed using atoms, and there are only finitely many elements up to automorphisms of $\A$. 
For the theory to make sense, we must assume that $\A$ is \emph{oligomorphic}, which means that $\A^d$ has finitely many orbits for every $d$. 
The simplest example of an oligomorphic atom structure is what we refer to as the \emph{equality atoms}; 
this is the structure with a countably infinite underlying set and no relations except for equality. 
This structure, like all oligomorphic structures over suitable vocabularies of relations, arises by applying a model-theoretic construction (the Fraïssé limit) to a well-behaved class of finite structures. 
Figure~\ref{fig:fraisse-limits} shows other examples of such structures. 

\begin{figure}
    \centering
    \begin{tabular}{lll}
   &  Class of finite structures & Fraïssé limit\\
    \toprule
    1.& finite sets with equality only & $(\N, =)$ \\
    2.& finite orders & $(\Q, <)$ \\
    3.& finite graphs & $(\text{Rado graph}, E)$ \\ 
    4.& (subsets of) finite $\FF_2$-vector spaces & $\FF_2 \oplus \FF_2 \oplus \FF_2 \oplus \cdots$ 
    \\
    \bottomrule  
\end{tabular}
\caption{Examples of Fraïssé limits.}\label{fig:fraisse-limits} 
\end{figure}

When the underlying atom structure is oligomorphic, the orbit-finite sets have a robust theory, which resembles in some ways the theory of finite sets. 
This theory was originally developed to describe regular languages over infinite alphabets, by considering orbit-finite versions of various automata models~\cite{bojanczykNominalMonoids2013,bojanczykAutomataTheoryNominal2014}, but it has since expanded to cover other models, such as orbit-finite Turing machines~\cite{bojanczykTuringMachinesAtoms2013} or orbit-finite constraint satisfaction problems~\cite{klin2015locally}.
There are also programming languages with data structures that can store orbit-finite sets~\cite{bojanczyk2012towards,bojanczyk2012imperative}, with working implementations~\cite{kopczynski2016lois, szynwelski-phd}. 
For a survey of the orbit-finite programme, we refer to the lecture notes~\cite{bojanczyk_slightly}. 

Some results about orbit-finite models do not depend on the choice of the atom structure, while others do. 
An example of the latter case arises for orbit-finite Turing machines~\cite{bojanczykTuringMachinesAtoms2013}. 
If the atoms are the Fraïssé limit of total orders, as in row~2 of Figure~\ref{fig:fraisse-limits}, then the orbit-finite version of the \textsc{p}$\stackrel{?}{=}$\textsc{np} question has the same answer as the classical version without atoms. 
On the other hand, for any of the other three rows of the table, \textsc{p}$\neq$\textsc{np} holds unconditionally in the orbit-finite setting due to problems with choice. 
(Failure of the axiom of choice in set theory with atoms was the original motivation of Fraenkel and Mostowski in the 1930s to study the first two rows of Figure 1; see e.g.~\cite[Ch.~4]{Jech}.)
The dependence on the underlying atom structure will play a prominent role in this paper. 

\paragraph*{Vector spaces}
One direction of the orbit-finite programme, motivated by the study of orbit-finite weighted automata, is focused on vector spaces~\cite{BFKM24}. 
In these spaces (taken over some fixed field), one can take linear combinations and apply automorphisms of the atom structure. 
We are interested in spaces that have an orbit-finite spanning set, which means that the entire space can be obtained from some finite subset by using atom automorphisms and linear combinations.
A prototypical example is the space $\Lin X$ that consists of formal linear combinations of elements from some orbit-finite set~$X$. 
The original application of these spaces was in automata theory, but they have also found applications in the study of  orbit-finite linear programming~\cite{ghosh2023orbit}, function spaces for orbit-finite sets~\cite{functionSpaces2024}, or  the analysis of two-party communication protocols over infinite alphabets~\cite{aliceBob}.

To be useful, the theory of orbit-finitely spanned vector spaces should have certain properties. 
For example, one would like to be able to represent these spaces in a finite way, or tackle algorithmic problems such as solving systems of linear equations. 
One rather modest requirement is that such spaces be closed under taking \emph{equivariant} subspaces 
(i.e.~subspaces closed under atom automorphisms as well as linear combinations): 
an equivariant subspace of an orbit-finitely spanned vector space should itself be orbit-finitely spanned. 
We do not know if this closure property holds in general, and we see it as an important open problem. 
It can be equivalently phrased in terms of ascending chains: 
is it true that every orbit-finitely spanned vector space is \emph{Noetherian}, meaning that one cannot find an infinite ascending chain of its equivariant subspaces? 
To the best of our knowledge, this question was first recognised by Camina and Evans~\cite[Q.~2]{CaminaEvans_91} who identified a sufficient condition for this, namely the existence of a ``nice ordering''. 
Using this condition they showed that certain vector spaces are Noetherian: notably, $\Lin \A$ over the ordered atoms (row~2 of Figure~\ref{fig:fraisse-limits}), $\Lin {\binom \A d}$ over the equality atoms (row~1), and a similar family of spaces over the bit-vector atoms (row~4). 

The above question was independently considered in~\cite[p.~21]{BFKM24} where it was conjectured that every oligomorphic structure has the \emph{ascending chain property}, meaning that all orbit-finitely spanned vector spaces over the structure are Noetherian.
In such a vector space, if descending chains as well as ascending chains of equivariant subspaces are all finite, by the Jordan--Hölder Theorem, there is some finite upper bound on the length of chains of its equivariant subspaces.
In that case, we say that the structure has the \emph{finite length property}.
This stronger property has been confirmed for:
\begin{itemize}
    \item \emph{The equality atoms.} 
    There are three independent proofs of finite length, in three different contexts: model theory~\cite[Thm.~3.9]{EvansRashwan_02}, representation theory~\cite[Prop.~6.1.6]{SamSnowden_15}, and orbit-finite set theory~\cite[Cor.~4.9]{BFKM24}. 
    The result from~\cite{SamSnowden_15} assumes that the underlying field is the complex numbers, while the other two do not restrict the choice of a field. 
    \item \emph{The ordered atoms.} 
    Here, the finite length property was proved in~\cite[Thm.~4.8]{BFKM24}. 
    There exist alternative methods to show the weaker ascending chain property of this structure.
    One is to exhibit an AZ enumeration \cite[Thm.~2.1]{RationalsHaveAZ} and derive the ``nice orderings'' from that~\cite{Evans_pm2}.
    Another is based on Hilbert's Basis Theorem~\cite[Thm.~27]{GhoshLasota_24}, which is independently established in~\cite[Ex.~3.2(c)]{LaudoneSnowden_23}, with well-quasi-orders as a key ingredient.
\end{itemize}

The finite length property is strictly stronger than the ascending chain property.
A prime example is the bit-vector atoms (which admit an AZ enumeration \cite[p.~143]{Hrushovski}):
even $\Lin \A$ does not have finite length over the two-element field, as found independently in~\cite[Thm.~2.7]{EvansGray_98} and~\cite[Thm.~4.16]{BFKM24}.
On the other hand, Evans~\cite[Rem.~1.3]{Evans_pm1} observed that for any Fraïssé limit over a finite relational vocabulary, failure of the finite length property over some finite field
would yield a counterexample to the conjecture of Thomas~\cite[p.~177]{thomas1991reducts} in model theory.
The known examples of failure of the finite length property are not good enough for this, since they use an infinite vocabulary (or functions).

\paragraph*{Our contributions} 
Until now, the finite length property has been a theory of two examples: the equality and the ordered atoms. 
We substantially improve this state of affairs, using two different techniques:
\begin{itemize}
    \item 
    In Theorem~\ref{thm:weak-smooth-approximation-finite-length}, we prove the finite length property for those structures $\A$~--~e.g.~the equality atoms and the bit-vector atoms~--~which admit what we call oligomorphic approximation, a relaxed version of smooth approximation known from model theory. 
    This is under an additional assumption that the underlying field has characteristic $0$.

    \item 
    In Theorem~\ref{thm:ordered-free-amalg-has-finite-length}, we prove the finite length property for those structures $\A$~--~e.g.~the ordered atoms~--~which arise as the generically ordered expansions of the Fraïssé limits of free amalgamation classes, over finite relational vocabularies of arity at most $2$. 
    Here we do not restrict the underlying field.
\end{itemize}
In particular, we may use either of these techniques to deduce the finite length property of the Rado atoms (row~3 in Figure~\ref{fig:fraisse-limits}), where even the weaker ascending chain property was not known before.

Most of the paper is devoted to introducing the background necessary to understand these results (Sections~\ref{sec:structures}--\ref{sec:spaces}) and to proving them (Sections~\ref{sec:characteristic-zero},~\ref{sec:free-amalg},~\ref{sec:cogs-turn} and Appendices~\ref{sec:appendix-symplectic},~\ref{sec:appendix-lachlan-woodrow},~\ref{sec:appendix-cogs}). 
In Section~\ref{sec:duals} we discuss some connections with function spaces and weighted automata, and in Section~\ref{sec:equivariant-subspaces} we  briefly discuss  applications to solving orbit-finite systems of equations.

%% file: src/Structures.tex
\section{Structures}\label{sec:structures}

In this section, we briefly recall some basic notions from model theory and describe the main examples of structures that we will consider in this paper.

Let us begin by fixing some terminology.
A \emph{vocabulary} is a set of relations, each with a specified positive arity. 
For example, the vocabulary of graphs contains one binary relation: the edge relation. 
(We do not include equality in the vocabulary, since it is automatically present.) 
A \emph{structure} $\A$ over a vocabulary consists of an underlying set, also denoted by $\A$, together with interpretations of relations from the vocabulary as actual relations on that set. 
An \emph{embedding} between two structures over the same vocabulary is an injective function between their underlying sets that preserves and reflects all relations. 
An \emph{isomorphism} is a surjective embedding.
An \emph{automorphism} of a structure is an isomorphism from the structure to itself; these form a group. 

Automorphisms of $\A$ act on tuples in $\A^d$ componentwise. 
When we speak of orbits in $\A^d$, we mean the orbits under this action of the automorphism group. 
For example, if $\A$ is a graph, then two pairs $(a_1,a_2)$ and $(b_1,b_2)$ are in the same orbit if and only if there is some automorphism of the graph that maps $a_1$ to $b_1$ and $a_2$ to $b_2$. 
In particular, the edge relation must be defined in the same way for both pairs.

All structures considered in this paper will be countable (or finite), and we will always want them to have finitely many orbits in every finite dimension as per the following definition.

\begin{definition}[Oligomorphic structure]\label{def:oligomorphic-structure}
    A structure $\A$ is \emph{oligomorphic} if $\A^d$ has finitely many orbits for every $d \in \set{1,2,\ldots}$.
\end{definition}

A relation on a structure $\A$~--~i.e.~a subset of $\A^d$~--~is called \emph{equivariant} if it is invariant under the action of the automorphism group. 
Equivalently, the relation is a union of orbits. 
If the structure is oligomorphic, then there are finitely many orbits to consider once the dimension $d$ is fixed, 
and therefore only finitely many equivariant relations. 
By the Ryll-Nardzewski Theorem~\cite[Thm.~7.3.1]{hodges1993model}, if the structure is oligomorphic and countable, 
then the equivariant relations are exactly those that can be defined (as subsets) in first-order logic~--~see for instance~\cite[Cl.~5.9]{bojanczyk_slightly}. 
In fact, the infinite structures that we consider in this paper will satisfy a stronger property: 
an equivariant relation will be definable not only by a first-order formula, but even by a quantifier-free one. 
This will be ensured by the additional homogeneity condition defined below. 
In the condition, a \emph{substructure} of $\A$ is any structure obtained by restricting $\A$ to some subset  of its underlying set; 
we do not distinguish between substructures and subsets.

\begin{definition}[Homogeneous structure] 
    A structure $\A$ is \emph{homogeneous} if every isomorphism between finite substructures of $\A$ 
    (i.e.~embedding of a finite substructure of $\A$ into $\A$) 
    extends to an automorphism of $\A$.
\end{definition}

In a homogeneous structure $\A$, an orbit in $\A^d$ consists of tuples that satisfy the same quantifier-free formulas~--~see for instance~\cite[Thm.~6.3]{bojanczyk_slightly}. 
Every homogeneous structure arises via a construction called the \emph{Fraïssé limit}~\cite[Sec.~7.4]{hodges1993model}, from a class of finite structures that satisfies certain closure properties (a so-called \emph{amalgamation class}). 
For the Fraïssé limit to be not only homogeneous, but also oligomorphic, we need to assume that the underlying class~--~consisting precisely of the finite structures that embed in the Fraïssé limit~--~has only finitely many non-isomorphic structures of each finite size. 
This assumption is automatically satisfied if the vocabulary contains only finitely many relations of each arity,
in particular if the vocabulary is finite.

Here are some of the important structures that we will consider in this paper.
They are all homogeneous and oligomorphic.
\begin{example}[Equality atoms] \label{ex:equality-atoms} 
    In this structure, the underlying set is the natural numbers and there are no relations other than equality. 
    Automorphisms are arbitrary permutations. 
    Two tuples in the same orbit necessarily have the same equality pattern,
    and conversely this condition is sufficient, which guarantees homogeneity.
    For instance, if two tuples $(a_1, a_2)$ and $(b_1, b_2)$ both have non-repeating entries, then the mapping $a_1 \mapsto b_1, a_2 \mapsto b_2$ can be extended first to a permutation of $\{a_1, a_2, b_1, b_2\}$, and then to one of the whole structure by mapping other elements identically.
    There is another orbit in dimension $d = 2$, defined by $x_1 = x_2$.
\lipicsEnd\end{example}

\begin{example}[Ordered atoms] \label{ex:order-atoms} 
    In this structure, the underlying set is the rational numbers, 
    equipped with the usual order~--~the vocabulary consists of this binary relation only.
    Automorphisms are order-preserving permutations. 
    Two tuples are in the same orbit if and only if they have the same order pattern. 
    For instance, for two tuples $(a_1, a_2)$ and $(b_1, b_2)$ such that $a_1<a_2$ and $b_1<b_2$, mapping $a_1 \mapsto b_1, a_2 \mapsto b_2$ can be extended to an automorphism by using Cantor's back-and-forth method, or by explicitly defining a monotone piecewise linear bijection.
    There are two other orbits in dimension $d = 2$, defined by $x_1 = x_2$ and $x_1 > x_2$.
\lipicsEnd\end{example}

\begin{example}[Vector atoms] \label{ex:vector-space-atoms} 
    Fix some finite field $\innerfield$, and let $\A$ be the vector space of countable dimension over $\innerfield$. 
    This vector space is seen as a structure over an infinite vocabulary containing, for every $d \in \set{1,2,\dots}$ and coefficients $\lambda_1, \dots, \lambda_d$ in $\innerfield$, the relation 
    \begin{align*}
        \setbuild{(a_1,\dots,a_d) \in \A^d}{ $\lambda_1 a_1 + \cdots + \lambda_d a_d = 0$ }.
    \end{align*}
    The vocabulary is defined so that automorphisms are the same thing as permutations that are linear maps. 
    Two tuples of the same length are in the same orbit if and only if they have the same linear dependencies. 
    By way of illustration, consider two tuples $(a_1, a_2)$ and $(b_1, b_2)$ each consisting of linearly independent entries.
    Extend $a_1, a_2$ to a basis for the subspace $V$ of $\A$ spanned by $a_1, a_2, b_1, b_2$; 
    do the same for $b_1, b_2$.
    This gives us an automorphism of $V$ with $a_1 \mapsto b_1, a_2 \mapsto b_2$, which can then be extended to an automorphism of $\A$ by extending these two bases of $V$ to bases of $\A$.
    In $\A^2$ there are $2 + |\innerfield|$ other orbits:
    $x_1 = 0 = x_2$;
    $x_1 = 0 \neq x_2$;
    and $x_1 \neq 0, x_2 = \lambda x_1$, where $\lambda$ ranges over the field $\innerfield$.
\lipicsEnd\end{example}

\begin{example}[Rado atoms]\label{ex:rado-graph} 
    The Rado graph is the Fraïssé limit of the class of finite undirected graphs. 
    Here, an undirected graph is viewed as a structure with one binary relation that is symmetric and irreflexive. 
    One explicit description of the Rado graph is as follows: its vertices are the natural numbers, and there is an edge between $n < m$ if the $n$-th least significant bit in the binary representation of $m$ is $1$.
    A more famous characterisation of the Rado graph is that if one randomly selects a graph with a countable set of vertices by independently including each possible edge with probability $1/2$, then with probability $1$ the resulting graph is isomorphic to the Rado graph. 
    
    As a Fraïssé limit, the Rado graph is homogeneous by construction.
    So two tuples are in the same orbit if and only if they have the same equality and adjacency patterns.
    In dimension $d = 2$ there are three orbits:
    the two coordinates can be equal, adjacent hence distinct, or distinct but non-adjacent.
\lipicsEnd\end{example}

\pagebreak

%% file: src/Sets.tex
\section{Orbit-finite sets}
\label{sec:orbit-finite-sets}
We shall now briefly explain the concept of orbit-finiteness.
We start with a countably infinite oligomorphic structure $\A$, as described in Section~\ref{sec:structures}, whose elements we call \emph{atoms}. 
These can be used to construct sets that are finite up to the automorphisms of $\A$, such as
\begin{align*}
   \myunderbrace{\A^2}{pairs}
   \qquad
   \myunderbrace{\setbuild{ (a,b) \in \A^2}{ $a \neq b$ }}{non-repeating pairs}
   \qquad
   \myunderbrace{\binom{\A}{2}.}{unordered pairs}
\end{align*}

There are several equivalent definitions of orbit-finiteness. 
Of these, we use one that is based on first-order interpretations~\cite[Sec.~4.3]{hodges1993model}. 
To construct an orbit-finite set, we proceed in three steps. 
In the examples above, this process is as follows:
start with all pairs of atoms, as in the first example; 
restrict to the equivariant subset of non-repeating pairs, as in the second example; 
and then take the quotient, as in the third example, by the equivariant equivalence relation that identifies two pairs if they differ only in the order of their elements.
The general definition is given below.

 \begin{definition}[Orbit-finite set]\label{def:orbit-finite-set}
    An \emph{orbit-finite set} over an oligomorphic structure $\A$ is any set that is obtained as follows: 
    \begin{enumerate}
        \item Start with a finite power $\A^d$ for some $d \in \set{1,2,\dots}$;
        \item Restrict it to an equivariant subset $X \subseteq \A^d$;
        \item Quotient $X$ under an equivariant equivalence relation. 
    \end{enumerate}
 \end{definition}

Let us justify the name ``orbit-finite'' in this definition. 
An orbit-finite set is equipped with an action of the automorphism group of the original structure $\A$, namely the action inherited from $\A^d$, suitably extended to the quotient. 
Under this action, the set has finitely many orbits: 
$\A^d$ has finitely many orbits by oligomorphicity, 
and the number of orbits can only go down when restricting to an equivariant subset and quotienting under an equivariant equivalence relation. 
  
There is another equivalent definition of orbit-finite sets, 
given in Definition~\ref{def:orbit-finite-set-alternative} below, that emphasises the role of the group action. 
In this definition, we use the following notion of support: 
if $X$ is a set equipped with an action of the automorphism group of $\A$, 
then a \emph{support} for an element $x \in X$ is any set $S \subseteq \A$ such that every automorphism $\pi$ of $\A$ satisfies
\begin{align*}
   \myunderbrace{\forall a \in S : \pi(a)=a}{action of $\pi$ on $\A$}
   \qquad \implies \quad
   \myunderbrace{\pi(x)=x.}{action of $\pi$ on $X$}
\end{align*}
\begin{definition}
   [Orbit-finite set, abstractly] \label{def:orbit-finite-set-alternative} An \emph{orbit-finite set} over an oligomorphic structure $\A$ is a set $X$ equipped with an action of the automorphism group of $\A$ such that: 
   \begin{bracketenumerate}
      \item every element has some finite support;
      \item there are finitely many orbits under the action.
   \end{bracketenumerate}
\end{definition}

Thanks to the underlying atom structure being infinite but oligomorphic, the two definitions above are equivalent 
in the sense that they describe the same sets, up to equivariant bijections: see~\cite[Thm.~5.12]{bojanczyk_slightly} and~\cite[Prop.~2.19]{BodirskyBodorMarimon_25}. 
Both definitions have their uses. 
Definition~\ref{def:orbit-finite-set} is more concrete, and it comes with a finite representation, which can be used for algorithms that process orbit-finite sets. 
On the other hand, some constructions (e.g.~disjoint unions, Cartesian products) are more naturally presented using the more abstract Definition~\ref{def:orbit-finite-set-alternative}. 

\flushbottom
\pagebreak

%% file: src/Spaces.tex
\section{Orbit-finitely spanned vector spaces}\label{sec:spaces}
We now introduce the main topic of this paper: vector spaces with an orbit-finite spanning set. 
We begin with spaces that have an orbit-finite basis; 
this special case has an elementary definition, and yet it will be the relevant case for almost all results of this paper. 

\begin{definition}[Orbit-finite-dimensional vector space]\label{def:orbit-finite-basis}
    $\Lin_\field{X}$ is, given an orbit-finite set $X$ and a field $\field$, the vector space of finite formal linear combinations of elements in $X$ over $\field$.
\end{definition}

This definition has two important parameters: the field $\field$ and the oligomorphic structure $\A$ over which $X$ is an orbit-finite set.
The spaces defined here have two kinds of structure: that of a vector space, and the action of the automorphism group of $\A$, via $\pi(\sum_i \lambda_i x_i) = \sum_i \lambda_i \pi(x_i)$. 
We will be interested in subsets that preserve both kinds of structure, 
i.e.~those that are closed under taking linear combinations as well as applying automorphisms of $\A$. 
Such subsets are called \emph{equivariant subspaces}.

\begin{example}\label{ex:length-one-dim-one}
    Let $\A$ be the equality atoms and let $\field$ be an arbitrary field. 
    As explained in~\cite[Ex.~4.2]{BFKM24}, the vector space $\Lin_\field \A$ has only three equivariant subspaces: 
    the zero subspace, the full space, and the subspace consisting of vectors whose coefficients add up to zero 
    (e.g.~$a - b$ and $a + b - 2 c$, with $a, b, c \in \A$ distinct). 
\lipicsEnd\end{example}

 Unfortunately, there is a price to pay for the elementary character of Definition~\ref{def:orbit-finite-basis}, 
 and that is the failure of certain closure properties. 
 In particular, these spaces are not closed (even up to equivariant linear bijections) under taking equivariant subspaces, or quotients by such spaces. 
 The problem with subspaces is already apparent in Example~\ref{ex:length-one-dim-one}, since the unique non-trivial subspace of $\Lin_\field \A$ does not have any equivariant basis, regardless of the underlying field~\cite[Ex.~6.1]{BFKM24}; 
 this subspace can also be realised as a quotient of $\Lin_\field \A^2$.
We therefore need a more general notion of a vector space, in the style of Definition~\ref{def:orbit-finite-set-alternative}.

\begin{definition}[Orbit-finitely spanned vector space, abstractly]\label{def:orbit-finitely-spanned-abstract}
    An \emph{orbit-finitely spanned vector space} over an oligomorphic structure $\A$ is a vector space equipped with an action of automorphisms of $\A$ such that: 
    \begin{bracketenumerate}
        \item
        vector addition and scalar multiplication are equivariant,
        \footnote{Equivalently, the group action $v \mapsto \pi(v)$ of each automorphism $\pi$ is linear.}

        \item 
        every vector is supported by a finite set of atoms;%
        \footnote{Equivalently, the application $(\pi, v) \mapsto \pi(v)$ is continuous, where the automorphism group is endowed with the topology of pointwise convergence whilst the vector space is endowed with the discrete topology. 
        See~\cite[Lemma~4.1.5]{hodges1993model}.}

        \item \label{item:orbit-finite-spanning-subset} 
        the space is spanned by some equivariant subset that is orbit-finite.%
        \footnote{Equivalently, the vector space is finitely generated as a module over the group ring of automorphisms.}
    \end{bracketenumerate}
\end{definition}

The spaces defined above include those of the form $\Lin_\field \A^d$, as $\A^d$ is orbit-finite. 
They are also easily seen to be closed under taking quotients by equivariant subspaces. 
Closure under taking equivariant subspaces is less obvious: 
one could imagine that condition~\eqref{item:orbit-finite-spanning-subset} above is violated by moving to an equivariant subspace. 
This closure property turns out to be closely related to the ascending chain property discussed in the introduction:

\begin{theorem}\label{thm:ACP}
    For any field $\field$ and oligomorphic structure $\A$, the following conditions are equivalent: 
    \begin{alphaenumerate}
        \item\label{item:sub-ofs-is-ofs}
        orbit-finitely spanned vector spaces are closed under taking equivariant subspaces; 
        
        \item\label{item:ofs-implies-Noetherian}
        for every orbit-finitely spanned vector space $V$, there are no infinite ascending chains of equivariant subspaces in $V$;

        \item\label{item:LinAd-is-Noetherian}
        for every $d \in \set{1,2,\dots}$, the vector space 
        $
            \Lin_\field \A^d
        $ does not have any infinite ascending chain of equivariant subspaces.
    \end{alphaenumerate}
    Furthermore, if these conditions hold, then the orbit-finitely spanned vector spaces are~--~up to equivariant linear bijections~--~precisely the vector spaces of the form $U/W$, 
    where $W \subseteq U$ are equivariant subspaces of $\Lin_\field \A^d$ and $d \in \set{1,2,\dots}$.
\end{theorem}
\begin{proof}[Proof sketch.]
    The equivalence of (\ref{item:ofs-implies-Noetherian}) and (\ref{item:sub-ofs-is-ofs}) is a classical result in module theory (see Appendix~\ref{sec:appendix-modules}), and the implication (\ref{item:ofs-implies-Noetherian}) $\Rightarrow$ (\ref{item:LinAd-is-Noetherian}) is immediate.
%
   For the ``furthermore'' part,
       take some orbit-finitely spanned vector space $V$. 
        By Definition~\ref{def:orbit-finite-set}, the spanning set can be obtained from some equivariant subset of $\A^d$ by quotienting under an equivariant equivalence relation. 
        This gives us a surjective equivariant linear map to $V$ from an equivariant subspace $U$ of $\Lin_\field \A^d$; 
        call its kernel $W$, so that we have a bijective equivariant linear map $V \leftarrow U/W$. 

    Finally, the implication (\ref{item:LinAd-is-Noetherian}) $\Rightarrow$  (\ref{item:ofs-implies-Noetherian})  follows from the ``furthermore'' part: 
    the lack of infinite chains is preserved by taking equivariant subspaces and images under equivariant linear maps. 
\end{proof}

We say that an atom structure $\A$ has the \emph{ascending chain property over a field $\field$} if any of the equivalent conditions in Theorem~\ref{thm:ACP} is satisfied. 
One interpretation of the theorem is that the ascending chain property is necessary for the theory of vector spaces to be well-behaved. 
In particular, thanks to the ``furthermore'' part, we get a result similar to the equivalence of Definitions~\ref{def:orbit-finite-set} and~\ref{def:orbit-finite-set-alternative} above, 
i.e.~a concrete representation of the orbit-finitely spanned vector spaces that can be used in algorithms (provided we know how to represent equivariant subspaces~--~see Section~\ref{sec:equivariant-subspaces}).

We are therefore interested in atom structures that have the ascending chain property. 
As it turns out, the techniques used in this paper will yield a stronger property, namely a finite bound on the length of chains:

\begin{definition}
    [Finite length property]
    \label{def:finite-length-property}
    An oligomorphic structure $\A$ has the \emph{finite length property over a field $\field$} if for every orbit-finite set $X$ over $\A$, there is a finite upper bound on the length of chains of equivariant subspaces 
    of $\Lin_\field X$.
    (The supremum of the chain lengths, finite or not, is called the \emph{length} of $\Lin_\FF X$.)
\end{definition}

In light of Theorem~\ref{thm:ACP} 
we could have used $\A^d$ instead of $X$ in the above definition.
As mentioned in the introduction, the finite length property can be strictly stronger than the ascending chain property, as witnessed by $\Lin_\field \A$ over the vector atoms from Example~\ref{ex:vector-space-atoms}. 
(Note that there are two fields involved, namely the finite field used to define the vector atoms $\A$ and the field $\field$ used to define $\Lin_\field \A$. 
In the counterexample, both fields are the two-element field.)
The finite length property was recently studied in~\cite{BFKM24}, where it was shown that the equality atoms (Example~\ref{ex:equality-atoms}) and the ordered atoms (Example~\ref{ex:order-atoms}) have this property over any field. 

The main contribution of this paper is to establish the finite length property of more structures. 
We will use two different techniques for this purpose.

%% file: src/Rado-cogless.tex
\section{Finite length in characteristic zero}
\label{sec:characteristic-zero}

In this section, we present the first of our two main results, 
which is a method for proving the finite length property assuming that the underlying field has characteristic zero. 
Under this assumption, we will establish the finite length property of the Rado atoms (Example~\ref{ex:rado-graph}) and of the vector atoms (Example~\ref{ex:vector-space-atoms}). 
These are new results. 
We also think that the proof itself, even when applied to get already known results for the equality atoms (Example~\ref{ex:equality-atoms}), is of independent interest and arguably simpler than previously known proofs.
The method that we use will work for all structures satisfying the following condition. 
 
\begin{definition}[Oligomorphic approximation]\label{def:oligomorphic-approximation} 
    We say that a homogeneous structure $\A$ has \emph{oligomorphic approximation} if, 
    for every $d \in \set{1,2,\ldots}$, 
    there exists a family $\Bclass$ of finite substructures of $\A$ such that: 
    \begin{bracketenumerate}
        \item \label{item:oligomorphic-approximation-embedding} 
        every finite substructure $S$ of $\A$ is a substructure of some $\B_S \in \Bclass$; and

        \item \label{item:oligomorphic-approximation-orbits} 
        there is a common finite upper bound, for all $\B \in \Bclass$, on the number of orbits in $\B^d$ with respect to automorphisms of $\B$.
    \end{bracketenumerate}
\end{definition}

This is related to the notion of \emph{smooth approximation} in model theory
introduced by Lachlan, developed in~\cite{KLM89} and studied in depth in~\cite{CherlinHrushovski_03}.
There one asks for a family $\Bclass$ that does not depend on $d$, 
and instead of (\ref{item:oligomorphic-approximation-orbits}) one requires that: 
\begin{description}
    \item[(2')]
    $\A$ is oligomorphic, and each $\B \in \Bclass$ is homogeneous. 
\end{description}

\begin{theorem}\label{thm:have-oligomorphic-approximation}
    The following structures have oligomorphic approximation: 
    \begin{romanenumerate}
        \item\label{item:smooth-implies-oligomorphic-approximation}   
        any smoothly approximated structure;\footnote{We work over the canonical vocabulary if the structure is not already homogeneous; 
        see \cite[p.~443]{KLM89}. This does not change the automorphism group or, consequently, the finite length property.}
        
        \item\label{item:eq-has-oligomorphic-approximation} 
        the equality atoms from Example~\ref{ex:equality-atoms};
        
        \item\label{item:vec-has-oligomorphic-approximation}
        the vector atoms from Example~\ref{ex:vector-space-atoms}, for any finite field $\innerfield$;        
        
        \item\label{item:rado-has-oligomorphic-approximation} 
        the Rado atoms from Example~\ref{ex:rado-graph}.

    \end{romanenumerate}
\end{theorem}

Before proving this, let us note that the ordered atoms (Example~\ref{ex:order-atoms}) do not have oligomorphic approximation.

\begin{non-example}\label{non-ex:weak-smooth-approximation}
Consider the rational numbers with the usual order. 
The finite substructures in this case are finite total orders, 
and already for dimension $d=1$, a finite total order of size $n$ has $n$ orbits. 
So no family $\Bclass$ satisfies both conditions (\ref{item:oligomorphic-approximation-embedding}) and (\ref{item:oligomorphic-approximation-orbits}).
\lipicsEnd\end{non-example}

\begin{proof}[Proof of Theorem~\ref{thm:have-oligomorphic-approximation}]
    For each of these structures, we exhibit a family $\Bclass$ satisfying the two conditions. In fact, in each case the family will not depend on the parameter $d$.
    \begin{romanenumerate}
        \item
        For a homogeneous structure $\A$ smoothly approximated by $\Bclass$, 
        we just use the family $\Bclass$.
        We only need to check condition~(\ref{item:oligomorphic-approximation-orbits}), so let $\B \in \Bclass$. 
        It is easy to check that, as observed in \cite[(II)]{KLM89}, two tuples in $\B^d$ are in the same orbit if and only if they are in the same orbit of $\A^d$.
        Hence the number of orbits in $\B^d$ is at most that of $\A^d$, which is a finite upper bound because $\A$ is oligomorphic.

        \item 
        For the equality atoms, we can choose $\Bclass$ to be all finite subsets, which clearly satisfies~(\ref{item:oligomorphic-approximation-embedding}). 
        Each of these is homogeneous: we can extend a bijection between subsets to a permutation.
        So $\Bclass$ satisfies~(2') and hence~(\ref{item:oligomorphic-approximation-orbits}) by the above argument.

        \item 
        For the vector atoms, we choose $\Bclass$ to be all subspaces of finite dimension. 
        Each of these is homogeneous since we can, by extending a linearly independent set to a basis, extend linear bijections between subspaces to a linear permutation. 
        Again $\Bclass$ satisfies~(\ref{item:oligomorphic-approximation-embedding}) and~(2') hence~(\ref{item:oligomorphic-approximation-orbits}).
        This argument applies to any finite field $\innerfield$.

        \item 
        The most interesting case is the Rado atoms, which are not smoothly approximated~\cite[Rem.~2.1.2]{CherlinHrushovski_03} (essentially because the finite homogeneous graphs~\cite[p.~100]{Gardiner1976} are not diverse enough). 
        The witness for oligomorphic approximation is a family of \emph{symplectic graphs}: see~\cite[Sec.~8.11]{GR01}.%
        \footnote{We are grateful to Ehud Hrushovski for drawing our attention to this construction.} 
        For every $n \in \set{1,2,\ldots}$ define a finite graph as follows. 
        The set of vertices is the vector space over the two-element field with basis 
        \begin{align*}
            \set{e_1,\ldots,e_n, f_1,\ldots,f_n}.
        \end{align*}
        Since the field has two elements, we can view vertices as subsets of this basis. In this graph, there is an edge between vertices $v$ and $w$ if and only if the sets 
        \begin{align*}
            \setbuild{ i \in \set{1,\ldots,n}}{ $e_i \in v$ and $f_i \in w$ }\\
            \setbuild{ i \in \set{1,\ldots,n}}{ $f_i \in v$ and $e_i \in w$ }
        \end{align*}
        have different sizes modulo two. 
        These graphs (embedded in all possible ways in the Rado graph) satisfy condition~\eqref{item:oligomorphic-approximation-embedding} from Definition~\ref{def:oligomorphic-approximation}, i.e.~every finite graph embeds in some symplectic graph~\cite[Thm.~8.11.2]{GR01}. 
        It is also not difficult~--~using techniques from~\cite[pp.~75--83]{Aschbacher_symplectic}~--~to prove condition~\eqref{item:oligomorphic-approximation-orbits}, 
        i.e.~that the number of orbits of $d$-tuples in symplectic graphs is uniformly bounded by a function of $d$ only; see Appendix~\ref{sec:appendix-symplectic}. \qedhere
    \end{romanenumerate}
\end{proof}

The main result of this section is the following theorem.

\begin{theorem}\label{thm:weak-smooth-approximation-finite-length}
    If $\A$ has oligomorphic approximation, then it has the finite length property over any field of characteristic $0$.
\end{theorem}

Combining Theorems~\ref{thm:have-oligomorphic-approximation} and~\ref{thm:weak-smooth-approximation-finite-length}, we can get the following results, both old and new, concerning the finite length property.

\begin{corollary}\label{cor:weak-finite-length}
    Over any field of characteristic $0$, the following structures have the finite length property: 
    (\ref{item:smooth-implies-oligomorphic-approximation}) all smoothly approximated structures;
    (\ref{item:eq-has-oligomorphic-approximation}) the equality atoms; 
    (\ref{item:vec-has-oligomorphic-approximation}) the vector atoms; 
    (\ref{item:rado-has-oligomorphic-approximation}) the Rado atoms.
\end{corollary}

As mentioned in the introduction, the finite length property was already known for the equality atoms~--~even over arbitrary fields. 
The results for the vector atoms and the Rado atoms are new. 
The assumption on characteristic zero is important, at least in the case of the vector atoms 
where the finite length property is known to fail over finite fields: see~\cite[Sec.~4.4]{BFKM24}. 
Later on in this paper, we will prove the result for the Rado atoms again using a different method that works for any field.

The rest of this section is devoted to proving Theorem~\ref{thm:weak-smooth-approximation-finite-length}.

\begin{proof}
    [Proof of Theorem~\ref{thm:weak-smooth-approximation-finite-length}]
Fix a structure $\A$ with oligomorphic approximation and a field of characteristic zero. 
Since the field is fixed, we omit the field subscript and write $\Lin X$ for linear combinations of elements in $X$ that use coefficients from that field. 
Fix some power $d \in \set{1,2,\ldots}$. 
Our goal is to show that $\Lin \A^d$ has the finite length property. 
For technical reasons, we apply the assumption on oligomorphic approximation not to $d$, but to $2d$, yielding  some class $\Bclass$ of finite structures that satisfies Definition~\ref{def:oligomorphic-approximation}.

We will now proceed in three steps. 
For convenience, we summarise these steps in Figure~\ref{fig:proof-summary-cogless}. 

    \begin{lemma}\label{lem:reduce-to-finite-substructures}
        For every $d \in \set{1,2,\ldots}$ we have 
        \begin{align*}
        \text{length of }\Lin \A^d
        \quad \leq \quad  \sup_{\B \in \Bclass} \text{\ length of }\Lin \B^d.
        \end{align*}
    \end{lemma}
    \begin{proof}
        Consider some chain of equivariant subspaces 
        \begin{align*}
            V_0 \subset V_1 \subset \dots \subset V_n \subseteq \Lin \A^d,
        \end{align*}
        where equivariance is with respect to automorphisms of $\A$. 
        For each $i \in \set{1,\ldots,n}$, choose some vector that is in $V_i$ but not in $V_{i-1}$. 
        Let $S$ be the (finite) set of atoms that appear in these chosen vectors, 
        and use~(\ref{item:oligomorphic-approximation-embedding}) to find some $\B \in \Bclass$ that contains all atoms from $S$. 
        Define  
        \begin{align*}
            W_i = V_i \cap \Lin \B^d.
        \end{align*}
        By homogeneity, every automorphism of $\B$ extends to an automorphism of $\A$, 
        and therefore the space $W_i \subseteq \Lin \B^d$ is equivariant with respect to automorphisms of $\B$. 
        The chain of $W_i$'s continues to be strictly growing, 
        since it contains vectors that witness the growth of the original chain. 
        Hence, the new chain forces the length of $\Lin \B^d$ to be at least $n$. 
    \end{proof}

\begin{figure}
\centering
\begin{tikzcd}[row sep=tiny]
	{\text{length of $\Lin \A^d$} } & \\[-1ex]
	& {\text{(Lemma~\ref{lem:reduce-to-finite-substructures})}} \\[-2ex]
	{\displaystyle\sup_{\B \in \Bclass} \text{length of $\Lin \B^d$}} \\[-2ex]
	& {\text{(Lemma~\ref{lem:dim-bounds-length})}} \\[-1ex]
	{\displaystyle\sup_{\B \in \Bclass} \text{dimension of } \lineqfun{\Lin \B^d}{\Lin \B^d}} \\[-2ex]
	& {\text{(Lemma~\ref{lem:bound-on-dim-for-finite-set})}} \\[-2ex]
	{\displaystyle\sup_{\B \in \Bclass} \text{number of orbits in $\B^{2d}$}} \\[-2ex]
	& {\text{(Definition~\ref{def:oligomorphic-approximation})}} \\[-2ex]
	\infty
	\arrow["\geq"{marking, allow upside down}, draw=none, from=3-1, to=1-1]
	\arrow["\geq"{marking, allow upside down}, draw=none, from=5-1, to=3-1]
	\arrow["{=}"{marking, allow upside down}, draw=none, from=7-1, to=5-1]
	\arrow["{>}"{marking, allow upside down}, draw=none, from=9-1, to=7-1]
\end{tikzcd}
\caption{Summary of the proof of Theorem~\ref{thm:weak-smooth-approximation-finite-length}.}\label{fig:proof-summary-cogless}
\end{figure}

    Thanks to the above lemma, it remains to show that the length of $\Lin \B^d$ is bounded by some number that depends only on $d$. 
    In our proof, this bound will be the number of orbits in $\B^{2d}$. 
    What we have gained by moving from $\A$ to $\B \in \Bclass$ is that our vector spaces now have finite (albeit unbounded) linear dimension, 
    and are acted upon by finite (albeit unbounded) groups. 
    This will let us leverage a well-known result from representation theory called Maschke's Theorem (see e.g.~\cite[Thm.~6.3]{RepTheoryTextbook}).
    In order to be more precise, recall that we write $V = V_1 \oplus V_2$ if each $v \in V$ is equal to $v_1 + v_2$ for some unique $(v_1, v_2) \in V_1 \times V_2$.

    \begin{fact}[Maschke's Theorem]
        Let $V$ be a finite-dimensional vector space over a field of characteristic zero, equipped with a linear action of a finite group $G$. 
        Then $V$ can be decomposed as 
        \begin{align*}
            V = V_1 \oplus \cdots \oplus V_n, 
        \end{align*}
        where each $V_i$ is an equivariant subspace (with respect to the action of $G$) 
        that has length $1$, 
        i.e.~the only equivariant subspaces of $V_i$ are the zero space and the full space $V_i$.
    \end{fact}

    We will use this theorem to bound the lengths of the vector  spaces $\Lin \B^d$ for $\B \in \Bclass$. 
    For two vector spaces $V$ and $W$ equipped with a linear action of the same group $G$, let us write 
    \[
        \lineqfun{V}{W} = 
        \left\{
            \psi : V \to W
            \ \middle\vert\
            \text{ $\psi$ is linear and }
            \forall g \in G, \forall v \in V : g(\psi(g^{-1}(v))) = \psi(v)\
        \right\}
    \]
    for the set of all those linear maps from $V$ to $W$ that are equivariant with respect to the action of $G$. 
   This set is closed under taking linear combinations, and therefore it can be seen as a vector space. 
    In particular, it is meaningful to talk about the dimension of this space. 
    \begin{lemma}\label{lem:dim-bounds-length}
        Let $V$ be a finite-dimensional vector space over a field of characteristic zero, 
        equipped with a linear action of a finite group $G$. 
        Then 
        \begin{align*}
            \text{length of\, $V$} 
            \quad \leq \quad 
            \text{dimension of } \lineqfun{V}{V}.
        \end{align*}
    \end{lemma}
    \begin{proof}
        Apply Maschke's Theorem, yielding a decomposition 
        \begin{align*}
        V = V_1 \oplus \cdots \oplus V_n,
        \end{align*}
        where the subspaces $V_1,\ldots,V_n$ are equivariant and of length $1$, with respect to the action of the group $G$. 
        The length is additive with respect to direct sums (see e.g.~\cite[Prop.~3.17]{RepTheoryTextbook}), i.e.
        \begin{align*}
            \text{length of $V$} 
            = \text{length of $V_1$} + \cdots + \text{length of $V_n$},
        \end{align*}
        so the length of $V$ is equal to $n$. 
        
        We will now consider the right-hand side of the inequality in the statement of the lemma.
        For every $i \in \set{1,\ldots,n}$, there is an equivariant linear map from $V$ to itself that 
        \begin{itemize}
            \item is the identity on $V_i$; and
            \item maps vectors from the other components to zero.
        \end{itemize}
        This gives us at least $n$ equivariant linear maps from $V$ to itself. 
        None of these maps can be spanned by the others, 
        and hence the dimension is no less than $n$. 
    \end{proof}

    Thanks to Lemmas~\ref{lem:reduce-to-finite-substructures} and~\ref{lem:dim-bounds-length}, 
    the length of the vector space $\Lin \A^d$ is bounded by the dimensions of the vector spaces
    \begin{align}\label{eq:endo-maps-b-d}
        \lineqfun{ \Lin \B^d}{\Lin \B^d}, 
        \tag{$*$}
    \end{align}
    where $\B$ ranges over the family $\Bclass$.
    To complete the proof of the theorem, it still remains to show that these dimensions are bounded by some number that depends only on $d$.
    We do so now.

    \begin{lemma}\label{lem:bound-on-dim-for-finite-set} 
        For every $\B \in \Bclass$, 
        the dimension of the vector space in~\eqref{eq:endo-maps-b-d} is 
        precisely the number of orbits in $\B^{2d}$.
    \end{lemma}
    \begin{proof}
        A linear map $\psi$ in the $\FF$-vector space~\eqref{eq:endo-maps-b-d} can be seen as a square matrix indexed by $\B^d$, 
        i.e.~a function
        \begin{align*}
            {\B^d \times \B^d} \to \field, 
            \quad
            \text{defined by } (x, y) \mapsto \lambda_{x,y} \text{ where } \psi(1 \cdot x) = \sum_{z} \lambda_{x,z} \cdot z.
        \end{align*}
        This function must be equivariant with respect to automorphisms of $\B$, 
        i.e.~for any automorphism $g$ and $x, y \in \B^d$ we must have $\lambda_{g(x),g(y)} = \lambda_{x, y}$.
        So to define such a function, we need to choose one element of $\field$ for each orbit of $\B^d \times \B^d$. 
        The dimension of the space in~\eqref{eq:endo-maps-b-d} is therefore equal to the number of orbits of $\B^{2d}$.
    \end{proof}
This number of orbits has a finite upper bound that depends only on $d$, by condition~(\ref{item:oligomorphic-approximation-orbits}) of Definition~\ref{def:oligomorphic-approximation}. 
This completes the proof of Theorem~\ref{thm:weak-smooth-approximation-finite-length}.
\end{proof}

The inequalities shown in Figure~\ref{fig:proof-summary-cogless} give us upper bounds on the length of $\Lin \A^d$. 
In the case of the equality atoms, this bound is the $2d$-th Bell number. 
In the case of the vector atoms from Example~\ref{ex:vector-space-atoms}, the bound is the number of linear dependency patterns for $2d$-tuples over a finite field. 
Such a pattern is described by: 
(a)~indicating a subset of the coordinates which is a basis for the tuple; 
and (b)~indicating the basis decompositions for the remaining coordinates. 
This can be done in at most $2^{(2d)^2}$ ways.
A similar upper bound can be obtained for the Rado graph by an analysis of orbits in symplectic graphs.

%% file: src/Duals.tex
\section{Function spaces and  weighted automata}\label{sec:duals}

The original motivation to introduce orbit-finitely spanned vector spaces in~\cite{BFKM24} was the study of weighted orbit-finite automata. 
In this section, we recall this motivation and discuss how it relates to our new results. 
This discussion also involves the issue of function spaces, arguably more fundamental, so we begin with that.
 
\subsection{Function spaces} 
Given two orbit-finitely spanned vector spaces $V$ and $W$ over the same atom structure, there are two natural ideas for a function space: 
the space of all linear maps from $V$ to $W$, and the subspace consisting of equivariant linear maps. 
As it turns out, the most relevant function space lies between them.

\begin{definition}[Finitely supported function space]\label{def:function-space}
    For two orbit-finitely spanned vector spaces $V$ and $W$, we define their \emph{finitely supported function space}, denoted by 
    \begin{align*}
     \linfsfun{V}{W},
    \end{align*}
    to be the vector space of linear maps $f$ that satisfy the following {finite support condition}: 
    there is some finite set of atoms $S \subset \A$ such that,
    for every automorphism $\pi$ of $\A$,
    \begin{align*}
        \forall a \in S : \pi(a)=a
        \quad \implies \quad 
        \forall v \in V : \pi(f(v)) = f(\pi(v)).
    \end{align*}
\end{definition}

This notion of finite support is the same as the one used in Section~\ref{sec:orbit-finite-sets}, 
here applied to the space of linear maps from $V$ to $W$. 
It is also the standard restriction used in the study of nominal sets~\cite[Thm.~2.19]{PittsAM:nomsns}.
As argued in~\cite[Sec.~8.3]{bojanczyk_slightly} on category-theoretic grounds, 
the finitely supported function space is the ``right'' function space.
We hence consider:

\begin{definition}[Function space property]\label{def:function-space-property}
    An atom structure $\A$ has the \emph{function space property} 
    if orbit-finitely spanned vector spaces over $\A$ are closed under taking finitely supported function spaces.
\end{definition}

The only structure known to have this property is the equality atoms~\cite[Thm.~8.16]{bojanczyk_slightly}.
If we only require function spaces of the form \[
    \linfsfun{V}{\field}
\] to be orbit-finitely spanned, we accordingly speak of the \emph{dual space property}. 
This weaker property is also satisfied by the ordered atoms~\cite[Cor.~6.8]{BFKM24} and, following \cite[Thm.~3.7]{Przybylek_2024}, by all $\omega$-stable oligomorphic structures~--~a strict subclass \cite[p.~440]{KLM89} of the smoothly approximated structures mentioned after Definition~\ref{def:oligomorphic-approximation}.
In contrast, the Rado atoms (which are not smoothly approximated but have the finite length property) fail to have even the dual space property, as we explain below.

\begin{example}\label{ex:no-function-spaces}
    Assume that $\A$ is the Rado atoms and the field $\FF$ is arbitrary~--~in particular, a field of characteristic zero as considered in Section~\ref{sec:characteristic-zero}.
    (A variant of this example for the two-element field was shown in~\cite[Ex.~6.9]{BFKM24}, when the finite length property had not been established yet.)
    We will show that
    \[
        \linfsfun{\Lin_\field \A}{\field}
    \]
    is not orbit-finitely spanned.
    This function space is easily seen to be in a linear equivariant bijection with the vector space
    \begin{align}\label{eq:dual-space}
        \fsfun{\A}{\field}
    \end{align}
    consisting of functions (not linear maps) from $\A$ to $\field$ that are finitely supported in the sense of Definition~\ref{def:function-space}.
    We will show that this space is not orbit-finitely spanned.
    
    For a finite set $S \subset \A$ of atoms, define a function $f_S : \A \to \field$ by
    \begin{align*}
        a \mapsto
        \begin{cases}
            1 & \text{if $a$ is a neighbour of all atoms in $S$}\\
            0 & \text{otherwise}.
        \end{cases}
    \end{align*}
    Define $V$ to be the subspace of~\eqref{eq:dual-space} that is spanned by the functions $f_S$, where $S$ ranges over finite sets of atoms.
    The spanning set $\set{f_S}_S$ is not orbit-finite, and
    as we prove in Appendix~\ref{sec:appendix-no-function-spaces}, 
    no orbit-finite spanning set can be found for this subspace.  
    But being orbit-finitely spanned is closed under taking subspaces, which follows from Theorem~\ref{thm:ACP} and the fact that the Rado atoms have the ascending chain property
    (by Theorem~\ref{thm:weak-smooth-approximation-finite-length} for fields of characteristic zero, and by the upcoming Corollary~\ref{cor:free-finite-length} for arbitrary fields). 
    Hence~\eqref{eq:dual-space} cannot be orbit-finitely spanned. 
\lipicsEnd\end{example}

\subsection{Weighted automata} 
We now explain how the issues with function spaces have an impact on the theory of weighted automata. 
There are several ways of defining these; 
we choose one viewing them as deterministic automata, in which the states are endowed with a vector space structure.  

\begin{definition}\label{def:orbit-finite-weighted}
    An \emph{orbit-finite weighted automaton} is given by: 
    \begin{itemize}
        \item $\Sigma$, an orbit-finite set (the alphabet);
        \item $Q$, an orbit-finitely spanned vector space (of states);
        \item $q_0$, an equivariant element of $Q$ (the initial state);
        \item $\delta$, an equivariant function of type $Q \times \Sigma \to Q$ (the transition function) that becomes a linear map from $Q$ to itself after fixing any  letter from $\Sigma$;
        \item $F$, an equivariant linear map of type $Q \to \field$ (the output map).
    \end{itemize}
    Such an automaton computes a function of type $\Sigma^* \to \field$, defined in the same way as for deterministic automata. 
\end{definition}

From the finite length property of the underlying atom structure, we can deduce decidability results about orbit-finite weighted automata. 
(Though, to be able to effectively represent the maps $\delta$ and $F$, we should assume that $Q$ has an orbit-finite basis.)
For example, a nondeterministic algorithm for non-equivalence is derived in~\cite[Sec.~5]{BFKM24}. 
In light of our present results, such an algorithm exists if we use the Rado graph as the atoms~--~or any other structure with the finite length property. 
Also, weighted automata can be minimized~\cite[Sec.~7]{BFKM24}; 
the bound for chains allows us to decide whether two input words are syntactically congruent.

However, certain other results on weighted automata depend on the function space property. 
Let us give one such example. 
Our definition of weighted automata is deterministic in the left-to-right direction. 
One could also imagine a model where the input word is processed right-to-left. 
Are these equivalent? 
If the atoms have the function space property and the ascending chain property, then one can introduce a symmetric model based on monoids (or more precisely, unital associative $\field$-algebras) to show that 
the left-to-right and right-to-left variants of weighted automata are equivalent: see, essentially,~\cite[Thm.~7.4]{BFKM24}. 
However, as we show in the following example, this equivalence can fail without the function space property.

\begin{example}
    Consider the Rado atoms and any field. 
    We shall prove that the left-to-right and right-to-left variants of Definition~\ref{def:orbit-finite-weighted} are not equivalent. 
    The counterexample is the characteristic function of the language ``the first letter is adjacent to all later ones'', 
    i.e.~the function $f$ defined by 
    \begin{align*}
        a_1 \cdots a_n \in \A^* \quad \mapsto \quad 
        \begin{cases}
            1 & \text{if $a_1$ is adjacent to all of $a_2,\ldots,a_{n}$}\\
            0 & \text{otherwise (e.g.~if $n = 0$).}
        \end{cases}
    \end{align*}
    We will show that this function is computed by a left-to-right orbit-finite weighted automaton, but not by a right-to-left one. 
    
    To prove this, for an input word $w \in \A^*$, we define two functions of type $\A^* \to \field$ as follows:
    \begin{align*}
        \myunderbrace{u \mapsto f(wu)}{left derivative}
        \quad \text{and} \quad 
        \myunderbrace{u \mapsto f(uw).}{right derivative}
    \end{align*}
    Using the usual Myhill--Nerode construction, 
    one can show that a function is computed by a left-to-right orbit-finite weighted automaton 
    if and only if (the space spanned by) the set of its left derivatives is orbit-finitely spanned~\cite[Thm.~6.19]{aliceBob}, 
    and similarly for its right derivatives. 

    The set of left derivatives is already orbit-finite: 
    there is one left derivative for each $a \in \A$, plus $f$ itself and one extra derivative for the always-zero function. 
    Hence the set of left derivatives is certainly orbit-finitely spanned.
    On the other hand, the set of right derivatives restricted to $\A$ is precisely the spanning set of the vector space $V$ from Example~\ref{ex:no-function-spaces}; 
    that space is not orbit-finitely spanned.
\lipicsEnd\end{example}

%% file: src/Rado-cogs.tex
\section{Finite length from free amalgamation with a generic order}\label{sec:free-amalg}
\renewcommand{\cal}{\mathcal}

As we saw in Non-example~\ref{non-ex:weak-smooth-approximation}, the ordered atoms do not have oligomorphic approximation. 
Nonetheless they do have the finite length property over any field~\cite[Thm.~4.8]{BFKM24}, not just over those of characteristic zero. 
In this section, we shall generalise that result to a wide class of structures. 
In particular, we will deduce the finite length property of the Rado atoms and its variants.

The structures that we consider here are Fraïssé limits satisfying certain conditions, which are defined by the underlying amalgamation classes.
We shall now state our assumptions and results, with proofs in Section~\ref{sec:cogs-turn}.

\begin{description}
    \item[Graph vocabulary.] 
    Consider a finite relational vocabulary $\sigma_0$ consisting of unary and binary relations only. 
    This allows us to talk about, in essence, directed graphs with coloured vertices and edges.

    \item[Free amalgamation class.] 
    Let $\Cclass_0$ be a \emph{free} amalgamation class of finite $\sigma_0$-structures. 
    The precise definition may be found in~\cite[Sec.~2.1]{Macpherson11}, 
    but informally it means that when we perform amalgamation, we do not need to glue together any new elements or introduce new relations.
    Here is a useful characterisation~\cite[Lem.~4.5]{SS20} of when an amalgamation class $\Cclass_0$ is free.
    Knowing that $\Cclass_0$ is closed under substructures and isomorphisms, 
    we see that it consists of all the finite $\sigma_0$-structures which do not embed any structure from $\Fclass$, 
    where $\Fclass$ consists of all the minimal (with respect to substructures) finite $\sigma_0$-structures that do not belong to $\Cclass_0$;
    we write \[
        \Cclass_0 = \Forb(\Fclass).
    \]
    That $\Cclass_0$ is a free amalgamation class means that in each $F \in \Fclass$, 
    every two elements $x, y$ are either equal or satisfy at least one of $R(x, y)$ and $R(y, x)$ for some binary relation $R \in \sigma_0$; in that case we will say that $x$ and $y$ are \emph{related}.
    Conversely, and conveniently, given a family $\Fclass$ of finite $\sigma_0$-structures where every two elements are related,
    the class $\Forb(\Fclass)$ of finite $\sigma_0$-structures is a free amalgamation class.    

\end{description}

\begin{example}\label{ex:free-amalg-equality}
    Let $\sigma_0$ be empty. 
    Then $\Cclass_0 = \Forb(\{\})$ is a free amalgamation class consisting of all finite pure sets.
\lipicsEnd\end{example}

\begin{example}\label{ex:free-amalg-graphs}
    Let $\sigma_0$ consist of a single binary relation $E$. 
    A (i)~\emph{graph}, (ii)~\emph{digraph}, (iii)~\emph{tournament} is a $\sigma_0$-structure where $E$ is irreflexive and (i)~symmetric, (ii)~antisymmetric, (iii)~antisymmetric and total.
    For instance, the ordered atoms can be viewed as a tournament once we interpret $E$ as the order relation.

    For more examples, consider $\sigma_0$-structures:
    \[\arraycolsep=20pt\def\arraystretch{1.2}
    \begin{array}{ccccc}
    	\xymatrix@C=15pt{\bigdot\ar@(d,r)[]}
	&
	\quad\xymatrix@C=15pt{\bigdot&\bigdot}
	&
	\xymatrix@C=15pt{\bigdot\ar[r]&\bigdot}
	&
	\xymatrix@C=15pt{\bigdot\ar@<.5ex>[r]&\bigdot\ar@<.5ex>[l]}
	&
	\vcenter{\xymatrix@C=15pt@R=6pt{\bigdot\ar@<.5ex>[dd]\ar@<.5ex>[rd] \\ &\bigdot\ar@<.5ex>[lu]\ar@<.5ex>[ld] \\ \bigdot\ar@<.5ex>[uu]\ar@<.5ex>[ru]}} \\
	\quad A & \quad B & C & K_2 & K_3
	\end{array}
    \]
    %
%
    Then:
    \begin{itemize}
        \item 
        $\Forb(\{A, C\})$ is a free amalgamation class and consists of all finite graphs;
        
        \item
        $\Forb(\{A, C, K_3\})$ is a free amalgamation class and consists of all finite triangle-free graphs;
        
        \item 
        $\Forb(\{A, K_2\})$ is a free amalgamation class and consists of all finite digraphs;

        \item 
        $\Forb(\{A, B, K_2\})$ consists of all finite tournaments. 
        It is an amalgamation class,
        but not a free one: 
        in the minimal forbidden structure $B$, the two elements are not related.
        \lipicsEnd
    \end{itemize}
\end{example}

\begin{description}
    \item[Irreflexivity.] 
    For technical reasons, we restrict attention to classes $\Cclass_0$ where all structures are {\em irreflexive}, i.e. such that if $R(x,y)$ then $x\neq y$ for each binary relation $R$.
    This does not lose generality, since we can replace every binary relation $R$ in the vocabulary with its irreflexive fragment $R_{\neq}(x,y)$ and a unary relation $R_=(x)$; see~\cite[Sec.~2.4]{Herwig1998} or~\cite[p.~121]{SS20}.
    All examples considered here are already irreflexive.
    
    \item[Generically ordered expansion.] 
    Let $\sigma$ consist of $\sigma_0$ together with a new binary relation $<$.
    Consider the class $\Cclass$ of $\sigma$-structures obtained from $\Cclass_0$ 
    by interpreting $<$ in any $\sigma_0$-structure there as any total order.
    Note that $\Cclass$ is an amalgamation class.
    Indeed, let $X, Y_1, Y_2$ be $\sigma$-structures in $\Cclass$ with $X \subseteq Y_1 \cap Y_2$.
    Then we can amalgamate $Y_1, Y_2$ over $X$ as $\sigma_0$-structures and as $\{<\}$-structures, 
    both using the disjoint union of $Y_1, Y_2$ over $X$ as the underlying set. 
    Superposing these relations gives a $\sigma$-structure in $\Cclass$, which is the desired amalgamation.
    (Because of the total order, this amalgamation in $\Cclass$ is not free except in trivial cases.)
    Denote the Fraïssé limit of $\Cclass$ by $\A$.    
\end{description}

Our main result of this section is that this ordered structure $\A$ has the finite length property over any field.
For clarity, we summarise our assumptions:
\begin{theorem}\label{thm:ordered-free-amalg-has-finite-length}
    Consider:
    \begin{itemize}
        \item $\sigma_0$\quad a finite, at-most-binary relational vocabulary;
        \item $\Cclass_0$\quad a free amalgamation class of finite, irreflexive $\sigma_0$-structures;
        \item $\Cclass$\quad all total orderings of all structures in $\Cclass_0$, over the vocabulary $\sigma_0 \uplus \{<\}$;
        \item $\A$\quad the Fraïssé limit of $\Cclass$~--~a homogeneous and oligomorphic $(\sigma_0 \uplus \{<\})$-structure.
    \end{itemize}
    Then the totally ordered structure $\A$, even with finitely many constants named (i.e.~added as unary relations to the vocabulary), has the finite length property over any field.
\end{theorem}

\begin{remark}
    We do not know how to drop the arity restriction from our proofs, 
    so it is interesting that Evans recently developed a Ramsey-theoretic approach in~\cite{Evans_pm1} which could be used to show that $\Lin_\field \A^2$, 
    for free amalgamation classes over vocabularies of arbitrarily high arity, has finite length.
    We do not know how to combine these results.

    We also do not know whether the finite length property of a general structure $\A$ implies the finite length property of the same structure with finitely many constants fixed.
    There is a quick, positive answer for the equality and the ordered atoms~\cite[Thm.~4.10]{BFKM24},
    but that argument relies on a property that these atoms have~\cite[Lem.~2.21]{BodirskyBodorMarimon_25} and the Rado atoms lack~\cite{no-constant-interpreted-in-Rado}.
\lipicsEnd\end{remark}

Before we prove Theorem~\ref{thm:ordered-free-amalg-has-finite-length} in Section~\ref{sec:cogs-turn}, let us state an easy consequence of it and give some examples (with more details in Appendix~\ref{sec:appendix-lachlan-woodrow}).
Note that the Fraïssé limit $\A_0$ of $\Cclass_0$ is a first-order reduct of $\A$.
To see this, notice that $\A$, when viewed as a $\sigma_0$-structure, embeds the same finite $\sigma_0$-structures as $\A_0$~--~namely, $\Cclass_0$.
Moreover, it follows from a back-and-forth argument~\cite[Lem.~7.1.4]{hodges1993model} that the $\sigma_0$-structure $\A$ is also homogeneous and therefore isomorphic to $\A_0$.
So we may assume that $\A$ and $\A_0$ have the same underlying set; we then have $\Aut(\A) \subseteq \Aut(\A_0)$, where $\Aut(X)$ denotes the group of automorphisms of $X$.
We thus call $\A$ the \emph{generically ordered expansion} of $\A_0$, or simply the \emph{ordered $\A_0$}.
\begin{corollary}\label{cor:free-finite-length}
    The reduct $\A_0$ of $\A$, even with finitely many constants fixed, has the finite length property over any field.
\end{corollary}
\begin{proof}
    A chain of subspaces in $\Lin_\FF \A_0^d = \Lin_\FF \A^d$ each supported by $S \subset \A_0$ is also a chain of subspaces supported by $S \subset \A$; 
    the latter has a bounded length by the theorem above.
\end{proof}

\begin{example}\label{ex:generically-ordered-equality}
    Continuing from Example~\ref{ex:free-amalg-equality},
    the ordered atoms (Example~\ref{ex:order-atoms}) are the generically ordered expansion of the equality atoms (Example~\ref{ex:equality-atoms}).
\lipicsEnd\end{example}

\begin{example}\label{ex:generically-ordered-graphs}
    Continuing from the first two items of Example~\ref{ex:free-amalg-graphs}, we obtain generically ordered expansions of the Rado graph and of \emph{Henson's triangle-free graph}.
    The ordered Rado graph was studied in~\cite{BPP15} along with its first-order reducts (e.g.~the Fraïssé limit of all finite tournaments). The ordered triangle-free graph will be studied in an extended example in Section~\ref{ex:extended-example}.
\lipicsEnd\end{example}

\begin{corollary}\label{cor:lachlan-woodrow}
    The following structures have the finite length property over any field:
    \begin{romanenumerate}
        \item the ordered atoms and the equality atoms;
        \item all three \cite{Lachlan1984} countable homogeneous tournaments;
        \item all of the countably many \cite{LachlanWoodrow_80} countable homogeneous graphs, including the Rado graph and Henson's triangle-free graph;
        \item uncountably many \cite[Thm.~2.4]{Henson1972} countable homogeneous digraphs obtained by forbidding tournaments.
    \end{romanenumerate}
\end{corollary}

\section{How the cogs turn: proof of Theorem~\ref{thm:ordered-free-amalg-has-finite-length}}\label{sec:cogs-turn}
To prove Theorem~\ref{thm:ordered-free-amalg-has-finite-length} we will proceed in several steps, 
with the general idea similar to~\cite[Sec.~4.1]{BFKM24}, 
but with significant new complications arising from the presence of non-trivial relations in $\A_0$.
Throughout we fix $\A$ from the statement of Theorem~\ref{thm:ordered-free-amalg-has-finite-length}.

\subsection{Orbits and projections}
To start with, let us view $\A^d$ as $\A^{\{1, \dots, d\}}$. 
More generally, it will be convenient to consider $\A^I$ for a finite totally ordered indexing set $I$ 
(say contained in $\Q$, so that we may take unions later).
Fix a finite support $S \subset \A$.
We shall say that a tuple $a\in\A^I$ is {\em ($S$-)ordered} if $a_i \not\in S$ for all $i$, and $a_i < a_j$ whenever $i < j$. 
Then the orbit $\cal O = \Aut(\A/S) \cdot a$ only contains $S$-ordered tuples, 
and we shall call the orbit $S$-ordered as well.%
\footnote{Here and in the following, $\Aut(\A/S)$ is the (sub)group of those automorphisms of $\A$ that fix every element of $S$. 
So ``$\Aut(\A/S)$-equivariant'' is synonymous with ``supported by $S$''.}
If $a$ is not $S$-ordered, by removing the entries that repeat or come from $S$ and reordering the rest, 
we can always find an $\Aut(\A/S)$-equivariant bijection from $\cal O$ to an $S$-ordered orbit.

To study the lengths of orbit-finitely spanned spaces, we may focus on a single ordered orbit at a time:
\begin{claim}\label{claim:reduction-to-one-orbit}
    The following are equivalent for every finite $S \subset \A$:
    \begin{alphaenumerate}
        \item 
        For each $d \in \{1, 2, \dots\}$, there is some number $l_d$ such that chains of $\Aut(\A/S)$-equivariant subspaces in $\Lin_\FF \A^d$ have length at most $l_d$;
        
        \item 
        For each $d \in \{1, 2, \dots\}$ and every $S$-ordered orbit $\cal O \subseteq \A^d$,
        $\Lin_\FF \cal O$ has finite length with respect to $\Aut(\A/S)$.
    \end{alphaenumerate}
\end{claim}
\begin{claimproof}
    Recall that $\A$ is oligomorphic: given any $S$ and $d$, 
    we know that $\A^d$ is in an $\Aut(\A/S)$-equivariant bijection with a finite disjoint union $\biguplus_i \cal O_i$ of $S$-ordered orbits (possibly in lower dimensions).
    Hence the length of $\Lin_\FF \A^d$, under the action of $\Aut(\A/S)$, equals \[
        \text{length of $\Lin_\FF(\biguplus_i \cal O_i)$} = \text{length of $\bigoplus_i \Lin_\FF \cal O_i$} = \sum_i \text{length of $\Lin_\FF \cal O_i$},
    \]
    which is finite if and only if each summand is finite.
\end{claimproof}

So fix $S$ and an $S$-ordered orbit $\cal O = \Aut(\A/S) \cdot o \subseteq \A^I$.
From here we take an inductive approach.
By $o |^J$ we mean the restriction of $o : I \to \A$ to $J \subseteq I$;
particularly, we will often write $o |^{-i}$ instead of $o |^{ I \setminus \{i\} }$.
The image $\cal O|^J$ of $\cal O$ under this projection agrees with $\Aut(\A/S) \cdot o |^J$ and is still ordered.
The function $(-)|^J$ lifts to a linear $\Aut(\A/S)$-equivariant map
\begin{align*}
    (-)|^J : \Lin_\FF \cal O &\to \Lin_\FF \cal O|^J, 
    \quad \text{defined by} \quad
    v \mapsto \sum_{a \in \cal O|^J} \sum_{b \in \cal O, b|^J = a} v(b) \cdot a.
\end{align*}
Many cancellations can occur under $(-)|^J$. 
Following the terminology used in~\cite[Eq.~(4)]{BFKM24}, 
by the \emph{balanced vectors} of $\Lin_\field \cal O$ we mean the $\Aut(\A/S)$-equivariant subspace below:
\[
    \Ker_\FF \cal O = \bigcap_{i \in I} \ker{ (-)|^{-i} }.
\]

\begin{claim}\label{claim:reduction-to-kernel}
    The following are equivalent for every finite $S \subset \A$:
    \begin{alphaenumerate}
        \item\label{item:LinO-finite-length}
        $\Lin_\FF \cal O$ has finite length with respect to $\Aut(\A/S)$ for every $S$-ordered orbit $\cal O$;

        \item\label{item:balancedLinO-finite-length}
        $\Ker_\FF \cal O$ has finite length with respect to $\Aut(\A/S)$ for every $S$-ordered orbit $\cal O$.
    \end{alphaenumerate}
\end{claim}
\begin{claimproof}
    That (\ref{item:LinO-finite-length}) implies (\ref{item:balancedLinO-finite-length}) is clear, as $\Ker_\FF \cal O \subseteq \Lin_\FF \cal O$.
    
    To prove the other implication, assume (\ref{item:balancedLinO-finite-length}) and let $\cal O \subseteq \A^I$.
    We proceed by induction on $|I|$.
    If $I = \emptyset$, then $\cal O$ must be the entire singleton $\A^\emptyset = \{ () \}$; 
    as $\Lin_\FF \cal O$ has no non-trivial subspaces (let alone finitely supported ones), it has length $1$.
    Now if $|I| \geq 1$, assemble all $|I|$ projection maps into a single map
    \begin{align*}
        \Lin_\FF \cal O \to \bigoplus_{i \in I} \Lin_\FF \cal O|^{-i},
        \quad \text{defined by} \quad
        v \mapsto ( v|^{-i} )_{i \in I},
    \end{align*}
    whose kernel is precisely $\Ker_\FF \cal O$.
    We thus have
    \[
        \text{length of $\Lin_\FF \cal O$}  - \text{length of $\Ker_\FF \cal O$}
        \leq \sum_{i \in I} \text{length of $\Lin_\FF \cal O|^{-i}$}, 
    \]
    which shows that the length of $\Lin_\FF \cal O$ is finite from the assumptions.
\end{claimproof}

As we will see shortly, there exist plenty of balanced vectors.

\subsection{Cogs}
From now on we will use a lightweight notation for combining tuples of atoms: 
for disjoint indexing sets $I$ and $J$ (contained in $\Q$), 
if $a\in\A^I$ and $b\in\A^{J}$ are both ordered, 
then $ab\in \A^{I\cup J}$ will denote their obvious combination. 
We only use this notation if this $ab$ is ordered. 
As an obvious example, for any $S$-ordered $a\in\A^I$ and $J\subseteq I$, we have $a|^{I\setminus J}a|^J=a$.

\begin{definition}\label{def:duo}
    Let $\cal O \subseteq \A^I$ be an $S$-ordered orbit. 
    An \emph{$\cal O$-duo} $a \parallel b$ consists of tuples $a, b \in \cal O$ such that:
    \begin{bracketenumerate}
        \item\label{item:cog-ai-bi} 
        $a_i < b_i$ for all $i\in I$;
        
        \item\label{item:cog-bi-aj}  
        $b_i < a_j$ for all $i<j\in I$;

        \item 
        for any binary $R$ in $\sigma_0$ (which we assumed to be irreflexive) and $i,j\in I$:
        \[
            R(a_i,b_j) \iff R(a_i,a_j) \stackrel{a,b\in\cal O}{\iff} R(b_i,b_j) \iff R(b_i,a_j).           
        \]
     \end{bracketenumerate}
\end{definition}
\begin{remark}\label{rem:duo}
Conditions (\ref{item:cog-ai-bi}) and (\ref{item:cog-bi-aj}) specify a total order on the $2|I|$ atoms in a duo.
Moreover, thanks to irreflexivity, 
each $a_i$ is unrelated to its counterpart $b_i$.
Further, given any $J \subseteq I$, the combined tuple $a|^{I\setminus J} b|^J$ satisfies the same relations as $a$ (and $b$), so it lies in $\cal O$. 
In particular, taking $J=\{i\}$, there is an automorphism $\pi_i$ that sends $a_i$ to $b_i$ and fixes all the other elements of $a$, $b$, and $S$.
\lipicsEnd\end{remark}

For the special case of the ordered atoms (Examples~\ref{ex:order-atoms} and~\ref{ex:generically-ordered-equality}) the following construction was studied in~\cite[p.~11]{BFKM24}, and it already appeared as early as in~\cite[p.~125]{Gray97} under the name of ``polytabloids''.

\begin{definition}
    Given an $\cal O$-duo $a \parallel b$, the corresponding \emph{$\cal O$-cog} is the vector
    \[
        a \between b =
        \sum_{J \subseteq I} (-1)^{|J|}
        (a|^{I \setminus J} b|^{J})
    \]
    in $\Lin_\FF \cal O$. 
    The linear span of all $\cal O$-cogs is denoted by $\Cog_\FF \cal O$.
\end{definition}
Note that, for a fixed $S$-ordered orbit $\cal O$, all $\cal O$-duos (hence all $\cal O$-cogs) are in the same $\Aut(\A/S)$-equivariant orbit.
As a result, $\Cog_\FF \cal O$ is an $\Aut(\A/S)$-equivariant subspace of $\Lin_\FF \cal O$ and it is generated by any single cog.

\begin{claim}\label{claim:cogs-are-balanced}
    $\Cog_\FF \cal O\subseteq \Ker_\FF \cal O$.
\end{claim}
\begin{claimproof}
    Let $\cal O \subseteq \A^I$, let $a\parallel b$ be an $\cal O$-duo, and take any $i \in I$.
    The subsets of $I$ come in pairs of $J$ and $J \cup \{i\}$, where $J\subseteq I \setminus \{i\}$.
    The two tuples $a|^{I\setminus J} b|^{J}$ and $a|^{I \setminus (J \cup \{i\})}b|^{J \cup \{i\}}$ are present in 
    $a\between b$ with the opposite coefficients, and they differ only on the $i$-th entry.
    Therefore they cancel out under $(-)|^{-i}$; hence $(a \between b )|^{-i} = 0$.
\end{claimproof}

It would be remiss not to address the fact that so far, we have not shown that $\cal O$-duos and $\cal O$-cogs even exist in general.
Let us rectify this by showing that they can be found in every non-zero $\Aut(\A/S)$-equivariant subspace of $\Lin_\FF{\cal O}$.

\subsection{Finding cogs}\label{subsec:finding-cogs}
We begin with a technical lemma, which combines the free amalgamation in $\Cclass_0$ and the generic order of $\A$.
\begin{lemma}\label{lem:free-fresh}
    Let $X, Y, \{z\} \subset \A$ be disjoint and finite.
    Then there is an automorphism $\tau \in \Aut(\A)$ such that
    \begin{bracketenumerate}
        \item\label{item:tau-fixes-X} 
        $\tau$ fixes every $x \in X$;

        \item\label{item:tau-frees-from-Yz} 
        $\tau(z)$ is unrelated to all $y \in Y$ and to $z$;
        
        \item\label{item:tau-adds-epsilon} 
        $\tau(z) > z$.
    \end{bracketenumerate}
\end{lemma}
\begin{proof}
Form the free amalgam in ${\Cclass}_0$ of $X\cup Y\cup\{z\}$ and $X\cup\{z\}$ over the common part $X$. The resulting structure can be presented as  $X\cup Y\cup\{z,z'\}$ for some new element $z'$. To make this a $\sigma$-structure, inherit the order on $X \cup Y \cup \{z\}$ from $\A$, and declare that $z < z'$, as well as $z' < a$ whenever $z<a$ for $a\in X \cup Y$. By homogeneity, $X \cup Y \cup \{z, z'\}$ embeds into $\A$ via some $f$ which is the identity on $X \cup Y \cup \{z\}$;
again by homogeneity, we may extend the embedding 
\[          x \in X \mapsto x, \quad z \mapsto f(z')     \] 
to some automorphism $\tau$ that satisfies (\ref{item:tau-fixes-X}), (\ref{item:tau-frees-from-Yz}), and (\ref{item:tau-adds-epsilon}).
\end{proof}

\begin{lemma}\label{lem:cog-fresh-single}
    Suppose $a\parallel b$ is an $\cal O$-duo, where $\cal O \subseteq \A^I$ is $S$-ordered.
    Given $z \in S$, 
    \begin{itemize}
        \item write $S' = S \setminus \{z\}$;
        \item let $j \not\in I$ be such that $a z \in \A^{I \cup \{j\}}$~--~thus ${\cal O'}=\Aut(\A / S') \cdot az \subseteq \A^{I \cup \{j\}}$~--~is $S'$-ordered;
        \item let $X \subset \A$ be any finite set containing $\{a_i, b_i \mid i \in I\} \cup S'$ but not $z$.
    \end{itemize}
    Denote $z'=\tau(z)$, where $\tau \in \Aut(\A/X)$ is afforded by Lemma~\ref{lem:free-fresh} (with an arbitrary $Y$). Then $az \parallel bz'$ is an $\cal O'$-duo.
\end{lemma}
\begin{proof}
    First, notice that $bz'\in {\cal O}'$ and that we have the required order relations with $z$ and $z'$. 
    The remaining condition (3) of Definition~\ref{def:duo}, for any binary $R$ in $\sigma_0$, splits into the following cases (and their symmetric counterparts):
\begin{itemize}
\item $R(a_i,b_{i'})\iff R(a_i,a_{i'})$ since $a \parallel b$ is an $\cal O$-duo;
\item $R(a_i,z')\iff R(a_i,z)$ since $\tau$ is an automorphism that fixes all $a_i$;
\item $R(a_i,z)\iff R(b_i,z)$ since $a, b \in \cal O$ and $z\in S$;
\item $R(z,z')$ and $R(z,z)$ are both false: $z'$ is unrelated to $z$ by Lemma~\ref{lem:free-fresh}, and $R$ is irreflexive. \qedhere
\end{itemize}
\end{proof}

Starting from an empty duo, we may apply Lemma~\ref{lem:cog-fresh-single} inductively and obtain:
\begin{lemma}\label{lem:cog-fresh-full}
    Let $\cal O \subseteq \A^I$ be an $S$-ordered orbit.
    Then any $a \in \cal O$ can be extended to an ${\cal O}$-duo $a\parallel b$.
\end{lemma}

As some $a \in \cal O$ always exists, it follows that Claim~\ref{claim:cogs-are-balanced} was not vacuous:
we now know $\Cog_\FF{\cal O}$, and hence $\Ker_\FF{\cal O}$, is non-zero~--~but barely so.
Indeed, as the result below shows, $\Cog_\FF{\cal O}$ admits no non-trivial $\Aut(\A/S)$-equivariant subspaces.
\begin{theorem} \label{thm:cogs-arise-everywhere}
    Any non-zero $\Aut(\A/S)$-equivariant subspace $V\subseteq \Lin_\FF \cal O$ contains $\Cog_\FF \cal O$.
\end{theorem}
\begin{proof}
    Pick any $v \in V$ and $a \in \cal O$ with $v(a)\neq 0$;
    it is enough to show that $V$ contains $a \between b$ for some $b\in \cal O$.
    Define:
    \[
        S^* = S \cup \{c_i \mid v(c) \neq 0, i \in I\} \setminus \{a_i \mid i \in I\} \supseteq S
    \]
    and put $\cal O^* = \Aut(\A / S^*) \cdot a \subseteq \cal O$~--~then $\cal O^*$ is $S^*$-ordered.
    By Lemma~\ref{lem:cog-fresh-full}, we can find $b \in \cal O^*$ such that $a \parallel b$ is an $\cal O^*$-duo and \emph{a fortiori} an $\cal O$-duo.
    Take the automorphisms $\pi_{i_1}, \dots, \pi_{i_d}$ from Remark~\ref{rem:duo}, where $i_1, \dots, i_d$ enumerate $I$.
    Now define $v^{(0)} = v$ and \[
        v^{(k)} = v^{(k-1)} - \pi_{i_k} v^{(k-1)}.
    \]
    Then each $v^{(k)}$ is in $V$.
   By induction on $k$, we have:
    \[
        v^{(k)} = \sum_{c \in C_k} \sum_{J \subseteq \{i_1, \dots, i_k\}} (-1)^{|J|} v(c) \left(\prod_{j \in J} \pi_j c\right),
    \]
    where
    \[C_k = \{ c \mid\ v(c) \neq 0 \text{ and } \{c_{i_1}, \dots, c_{i_k}, \dots, c_{i_d}\} \supseteq \{a_{i_1}, \dots, a_{i_k}\}\ \}.
    \]
    But $\{c_{i_1}, \dots, c_{i_d}\} \supseteq \{a_{i_1}, \dots, a_{i_d}\}$ means that $c = a$, so $C_d=\{a\}$ and $\frac{1}{v(a)}v^{(d)}=a\between b$.
\end{proof}

\begin{corollary}\label{cor:cog-simple}
    $\Cog_\FF \cal O$ has length $1$.
\end{corollary}

In light of Claims~\ref{claim:reduction-to-one-orbit},~\ref{claim:reduction-to-kernel} and~\ref{claim:cogs-are-balanced}, 
for the finite length property it is enough to prove that $\Ker_\FF \cal O \subseteq \Cog_\FF \cal O$. 
In words, we need to show that every balanced vector in $\Lin_\FF{\cal O}$ is a linear combination of $\cal O$-cogs.
Before stating this (as Theorem~\ref{thm:cog-span-generally}), let us illustrate the key ideas of a proof on an example.

\subsection{Spanning by cogs: an extended example}
\label{ex:extended-example}
Let $\A_0$ be the universal triangle-free (undirected) graph, introduced by Henson~\cite{Henson1971}, and let $\A$ be its generically totally ordered version. 
Consider nine atoms $\{a,\ldots,i\}$, ordered by~$<$ alphabetically, with the edge relation as shown here:
\[\scalebox{0.8}{
\rotatebox{7}{
    \xymatrix@R=6pt@C=15pt{
    \rotatebox{-7}{$h$} & & & & & & \rotatebox{-7}{$i$} \\
    \\
    & & & \rotatebox{-7}{$g$}\ar@{-}[uulll]\ar@{-}[uurrr] \\
    & \rotatebox{-7}{$e$} & & & & \rotatebox{-7}{$f$} \\
    & & \rotatebox{-7}{$c$}\ar@{-}[ul]\ar@{-}[uur] & & \rotatebox{-7}{$d$}\ar@{-}[uul]\ar@{-}[ur] \\
    \rotatebox{-7}{$a$}\ar@{-}[rrrrrr]\ar@{-}[uuuuu]\ar@{-}[uur] & & & & & & \rotatebox{-7}{$b$}\ar@{-}[uul]\ar@{-}[uuuuu]
    }
}
}\]
This graph is drawn so that the total order of the atoms corresponds to the vertical order.

\pagebreak

Putting $S=\emptyset$ and $|I|=2$, let ${\cal O}$ be the ordered orbit of pairs of atoms which are adjacent. 
Consider the following vector:
\[
    v = ah - ae + ce - cg + dg - df + bf - bi + gi - gh \in \Lin_\FF{\cal O}.
\]
This can be graphically presented as the following graph:
\[\scalebox{0.8}{
\rotatebox{7}{
    \xymatrix@R=6pt@C=15pt{
    \rotatebox{-7}{$h$} & & & & & & \rotatebox{-7}{$i$} \\
    \\
    & & & \rotatebox{-7}{$g$}\ar@[blue][uulll]\ar@[red][uurrr] \\
    & \rotatebox{-7}{$e$} & & & & \rotatebox{-7}{$f$} \\
    & & \rotatebox{-7}{$c$}\ar@[red][ul]\ar@[blue][uur] & & \rotatebox{-7}{$d$}\ar@[red][uul]\ar@[blue][ur] \\
    \rotatebox{-7}{$a$}\ar@{-}[rrrrrr]\ar@[red][uuuuu]\ar@[blue][uur] & & & & & & \rotatebox{-7}{$b$}\ar@[red][uul]\ar@[blue][uuuuu]
    }
}
}\]
where edges with coefficient $+1$ are marked as red, and with $-1$ as blue. 
The arrows on the chosen edges remind us that the pairs in $\cal O$ are ordered, but this is mere decoration: 
the definition of $\cal O$ means that all arrows must point upwards.

Note that $v$ is balanced. 
Graphically, this means that every atom has as many outgoing red edges as outgoing blue edges, and as many incoming red edges as incoming blue edges.

It is easy to draw $\cal O$-cogs in this way. 
Given some additional atom $z>h$ which is adjacent to $a$ and $g$ but not to $h$, the $\cal O$-cog $ah\between gz$ can be drawn as:
\[\rotatebox{15}{
    \xymatrix{
    \rotatebox{-15}{$h$} & \rotatebox{-15}{$z$} \\
    \rotatebox{-15}{$a$}\ar@[red][u]\ar@[blue][ur] & \rotatebox{-15}{$g$}\ar@[blue][ul]\ar@[red][u]  
    }
}\]

We would like to present $v$ as a sum of such ${\cal O}$-cogs. 
Some additional atoms must be used for that, as no four atoms among the original nine form an $\cal O$-duo.
It would be very convenient to introduce a single new atom $z$ to form all the $\cal O$-duos that we will use. 
(Such a new atom is not necessary if $\A$ is the ordered atoms, as then we can simply take $z$ to be the largest atom present in the vector; 
see also the proof of~\cite[\text{Cl.~4.7}]{BFKM24}.)
We can naïvely require $z$ to be:
\begin{itemize}
    \item 
    larger than every atom in $v$,

    \item 
    adjacent to every atom that is a source of a directed edge in $v$ 
    (equivalently: that occurs as the first component of a pair in $v$), and
    
    \item 
    not adjacent to any atom that is a target of a directed edge in $v$ 
    (equivalently: that occurs as the second component of a pair in $v$).
\end{itemize}
However, such a $z$ does not exist in the ordered triangle-free graph $\A$. 
There are two problems:
\begin{itemize}
    \item 
    The atom $g$ occurs both as the first and as the second component in pairs present in $v$. 
    Our specification of whether $z$ is adjacent to $g$ is therefore inconsistent.
    
    \item 
    Atoms $a$ and $b$ both occur as first components in $v$, and they are adjacent in $\A$. 
    As a result, an atom $z$ as prescribed would create a triangle $abz$ in $\A$, which is forbidden.
\end{itemize}
We resolve such {\em conflicts} by considering auxiliary atoms $g'>g$ and $b'>b$, with just enough edges to make $gh\parallel g'i$ and $bf\parallel b'i$ valid ${\cal O}$-duos. 
Specifically, we postulate edges $E(g', h), E(g', i), E(b', f), E(b', i)$ and no more. 
Such atoms exist by Lemma~\ref{lem:free-fresh}. 
We then define:
\[
    v' = v + (gh \between g'i) - (bf \between b'i)
\]
which can be drawn as:
\[\scalebox{0.8}{
\rotatebox{7}{
    \xymatrix@R=6pt@C=15pt{
    \rotatebox{-7}{$h$} & & & & & & \rotatebox{-7}{$i$} \\
    \\
    & & & \rotatebox{-7}{$g$}\ar@{-}[uulll]\ar@{-}[uurrr] & \rotatebox{-7}{$g'$}\ar@[blue]@/_2ex/[uullll]\ar@[red]@/_1ex/[uurr] \\
    & \rotatebox{-7}{$e$} & & & & \rotatebox{-7}{$f$} \\
    & & \rotatebox{-7}{$c$}\ar@[red][ul]\ar@[blue][uur] & & \rotatebox{-7}{$d$}\ar@[red][uul]\ar@[blue][ur] \\
    \rotatebox{-7}{$a$}\ar@{-}[rrrrrr]\ar@[red][uuuuu]\ar@[blue][uur] & & & & & & \rotatebox{-7}{$b$}\ar@{-}[uul]\ar@{-}[uuuuu] & \rotatebox{-7}{$b'$}\ar@[red]@/_1ex/[uull]\ar@[blue]@/_1ex/[uuuuul]
    }
}
}\]
Now an atom $z$ as postulated above does not create any triangles:
\[\scalebox{0.8}{
\rotatebox{7}{
    \xymatrix@R=6pt@C=15pt{
    & & & \rotatebox{-7}{$z$}\ar@{-}@/_4ex/[llldddddd]\ar@{-}[lddddd]\ar@{-}[rddddd]\ar@{-}@/_-0.5ex/[rddd]\ar@{-}@/_-3ex/[rrrrdddddd] \\
    \rotatebox{-7}{$h$} & & & & & & \rotatebox{-7}{$i$} \\
    \\
    & & & \rotatebox{-7}{$g$}\ar@{-}[uulll]\ar@{-}[uurrr] & \rotatebox{-7}{$g'$}\ar@[blue]@/_2ex/[uullll]\ar@[red]@/_1ex/[uurr] \\
    & \rotatebox{-7}{$e$} & & & & \rotatebox{-7}{$f$} \\
    & & \rotatebox{-7}{$c$}\ar@[red][ul]\ar@[blue][uur] & & \rotatebox{-7}{$d$}\ar@[red][uul]\ar@[blue][ur] \\
    \rotatebox{-7}{$a$}\ar@{-}[rrrrrr]\ar@[red][uuuuu]\ar@[blue][uur] & & & & & & \rotatebox{-7}{$b$}\ar@{-}[uul]\ar@{-}[uuuuu] & \rotatebox{-7}{$b'$}\ar@[red]@/_1ex/[uull]\ar@[blue]@/_1ex/[uuuuul]
    }
}
}\]
and it is easy to calculate:
\[
v' = (ah \between g'z)  - (ae \between cz) - (cg \between dz) + (b'f \between dz) -(b'i \between g'z)
\]
which presents $v = v' - (gh \between g'i) + (bf \between b'i)$ as a linear combination of $\cal O$-cogs.

\subsection{All those equivariant subspaces}\label{sec:equivariant-subspaces}
Inspired by the above example, we shall now assert in somewhat greater generality:

\begin{theorem}\label{thm:cog-span-generally}
    Every $v \in \Ker_\EE \cal O$ can be written as
    \[
        v = \sum_{a \parallel b} \lambda_{a \parallel b} \cdot a \between b
    \]
    with $\lambda_{a \parallel b} \in \EE$,
    where $\EE$ is any subspace of some $\FF^n$.
\end{theorem}

\begin{remark}\label{rem:FF-EE}
We must tread carefully here, as $\EE$ is not a field.
For example $\EE$ can be $\{ (\kappa, -\kappa) \mid \kappa \in \FF \}$ in $\FF^2$. 
Then $\Lin_\EE \cal O$ is an equivariant subspace of $\Lin_{\FF^2} \cal O$; 
the latter is more traditionally viewed as $\Lin_\FF \cal O \oplus \Lin_\FF \cal O$.

Given $\lambda \in \EE$ and $a \in \cal O$, we need to understand
\(
    \lambda \cdot a
\)
as a formal expression for an element of $\Lin_\EE \cal O$, rather than the result of a scalar multiplication.
Accordingly, we need to redefine $\Cog_\EE\cal O$ to be spanned by formal expressions $\lambda\cdot(a \between b)$ for $\lambda\in\EE$. 
We still have $\Lin_{\FF^n} \cal O \supseteq \Lin_\EE \cal O \supseteq \Ker_\EE \cal O \supseteq \Cog_\EE \cal O$ as equivariant spaces over $\FF$.
\lipicsEnd\end{remark}

The proof of Theorem~\ref{thm:cog-span-generally} follows the lines of the example in Section~\ref{ex:extended-example}, by carefully removing atom conflicts from a vector before inventing a fresh atom that creates enough duos and cogs. Technical details are given in Appendix~\ref{sec:appendix-cogs}.

Putting $\EE = \FF$ in Theorem~\ref{thm:cog-span-generally}, we have~--~as discussed at the end of Section~\ref{subsec:finding-cogs}~--~established Theorem~\ref{thm:ordered-free-amalg-has-finite-length}.
But by allowing $\EE$ to be finite-dimensional vector spaces, we can now describe all equivariant subspaces of an orbit-finite-dimensional vector space: 
we need only look at local sums of coefficients.

\begin{theorem}[Compare with~{\cite[Cor.~3.17]{Gray97},~\cite[Thm.~15]{HofmanLerouxTotzke_17}, \cite[Thm.~3.4]{HofmanRozycki_2022}, \cite[Sec.~6]{GhoshHofmanLasota_22}, \cite[Thm.~1.4]{Evans_pm1}}]\label{thm:coeff-approximation}
    Let $\A$ be as in Theorem~\ref{thm:ordered-free-amalg-has-finite-length} and fix $d \in \{1, 2, \dots\}$. 
    Then there exists a finite family of equivariant linear maps of the form
    \[
        \restriction_i : \Lin_\FF \A^d \to \Lin_{\FF^{n_i}} \cal O_i,
    \]
    where $\cal O_i$ is an equivariant ordered orbit of $\A^{d_i}$ with $d_i \leq d$, 
    such that every equivariant subspace $W \subseteq \Lin_\FF \A^d$ is equal to
    \[
        \{v \in \Lin_\FF \A^d \mid\ \forall i, \forall a \in \cal O_i : v{\restriction_i}(a) \in \EE_i\ \}
    \]
    with the finite-dimensional subspaces $\EE_i \subseteq \FF^{n_i}$ given by
    \[
        \EE_i = \{w {\restriction_i}(b) \mid\ w \in W, b \in \cal O_i\ \}.
    \]
\end{theorem}

\begin{example}
    Let $\A$ be the ordered atoms. 
    For $d = 1$ there are two maps, $\restriction_{1} : \Lin_\FF \A \to \Lin_\FF \A$ and $\restriction_{2} : \Lin_\FF \A \to \FF$, given by  
    \[
        v {\restriction_1} (a) = v(a), \qquad
        v {\restriction_2} () = \sum_{x} v(x).
    \]
    For an equivariant vector space $V \subseteq \Lin_\FF \A$, the subspace $V {\restriction_1}(\A) \subseteq \FF$ is either $\{0\}$ or $\FF$.
    \begin{itemize}
        \item In the first case, $V$ must be the zero space.
        \item In the second case, we similarly distinguish two cases for $V {\restriction_2}() \subseteq \FF$:
        \begin{itemize}
            \item if it is the entire $\FF$, then $V$ must be the full space $\Lin_\FF \A$ by Theorem~\ref{thm:coeff-approximation};
            \item if it is $\{0\}$, then $V$ must be the zero-sum space (which exists and is spanned by $\{a - b \mid a, b \in \A\}$) again using Theorem~\ref{thm:coeff-approximation}.
        \end{itemize}
    \end{itemize}
    So $\Lin_\FF \A$ has the same structure as the space over the equality atoms in Example~\ref{ex:length-one-dim-one}.

    For $d = 2$, there are three maps, one of which is of the form $\restriction_2 : \Lin_\FF \A^2 \to \Lin_{\FF^5} \A$. 
    By considering subspaces of $\FF^5$, when the field is infinite we can find infinitely many equivariant subspaces in $\Lin_\FF \A^2$.
    Note that the length of $\Lin_\FF \A^2$ is still finite (and equal to $2^1 + 2^2 + 2^2$, as we will see below).
\lipicsEnd\end{example}

Going through the finite-dimensional subspaces of the $\FF^{n_i}$'s gives a complete list of the equivariant subspaces of $\Lin_\FF \A^d$, by Theorem~\ref{thm:coeff-approximation} (with proof in Appendix~\ref{sec:appendix-eqsubsp}), but note that it is far from being irredundant.
In particular, to bound the length of $\Lin_\FF \A^d$ from below, we still need to exhibit a long chain of equivariant subspaces.
It is fortunately straightforward to adapt \cite[Cor.~4.12]{BFKM24} and establish an exact length:
\begin{corollary}\label{cor:2d-length}
    The length of $\Lin_\FF \cal O$, where $\cal O \subseteq \A^d$ is an ordered orbit, is precisely $2^d$.
\end{corollary}

As described in~\cite[Sec.~9]{GhoshLasota_24}, results like Theorem~\ref{thm:coeff-approximation} allow us to decide the solvability of orbit-finite systems of linear equations. 
We briefly repeat the argument.
Checking whether the system $\mathbf{A} \mathbf{x} = \mathbf{b}$ admits a solution amounts to checking whether $\mathbf{b}$ is spanned by the columns of $\mathbf{A}$.
In the orbit-finite setting, we assume that the rows and columns are indexed by $\A^d$ and $\A^{d'}$, that each column has finitely many non-zero entries, and that the definition $j \mapsto \mathbf{A}_{-, j}$ of columns is equivariant.
That is, we ask whether $\mathbf{b} \in \Lin_\FF \A^d$ lies in the span of the $\mathbf{A}_{-, j}$'s. 
By Theorem~\ref{thm:coeff-approximation}, it suffices to check whether $\mathbf{b} {\restriction_i} (\cal O_i) \subseteq \mathbf{A}_{-, j_1} {\restriction_i} (\cal O_i) + \cdots + \mathbf{A}_{-, j_r} {\restriction_i} (\cal O_i)$ for all $i$, in the finite-dimensional spaces $\FF^{n_i}$, having chosen orbit representatives $j_1, \dots, j_r \in \A^{d'}$.

%% file: src/Conclusion.tex
\section{Conclusion}
With Theorems~\ref{thm:weak-smooth-approximation-finite-length} and~\ref{thm:ordered-free-amalg-has-finite-length} and their Corollaries~\ref{cor:weak-finite-length} and~\ref{cor:lachlan-woodrow}, 
we have extended the finite length property far beyond equality and ordered atoms, the two examples considered in~\cite{BFKM24}. 

We finish the paper by reiterating some general questions that nonetheless remain open:
\begin{itemize}
    \item \cite[Q.~2]{CaminaEvans_91}
    Does every oligomorphic structure have the ascending chain property? 
    
    \item \cite[Q.~1.2]{Evans_pm1}
    Does every structure that is homogeneous over a finite relational vocabulary have the finite length property?
\end{itemize}
We also ask a new one, prompted by the limited but growing stock of examples: 
\begin{itemize}
    \item 
    Does every oligomorphic structure have the finite length property over fields of characteristic $0$?
\end{itemize}
We do not even know the answer to these questions for some concrete and well-studied Fraïssé limits, 
notably the universal partial order or the countable atomless Boolean algebra.

%% file: src/appendix-acp.tex
\section{Proof of Theorem~\ref{thm:ACP}}\label{sec:appendix-modules}
    Let us recall the proof of (\ref{item:sub-ofs-is-ofs}) $\Leftrightarrow$ (\ref{item:ofs-implies-Noetherian}) from classical module theory. 
    For (\ref{item:ofs-implies-Noetherian}) $\Rightarrow$ (\ref{item:sub-ofs-is-ofs}), we build an orbit-finite spanning set for the subspace greedily, by adding orbit after orbit of vectors. 
    Thanks to (\ref{item:ofs-implies-Noetherian}), this must stabilise. 
    For (\ref{item:sub-ofs-is-ofs}) $\Rightarrow$ (\ref{item:ofs-implies-Noetherian}), we use (\ref{item:sub-ofs-is-ofs}) to get an orbit-finite spanning set for the union of the chain from (\ref{item:ofs-implies-Noetherian}), 
    and then we observe that orbits from this set can only be added finitely often in the chain. 

    Let us now prove the ``furthermore'' part. 
    \begin{itemize}
        \item 
        In one direction, consider equivariant subspaces $W \subseteq U \subseteq \Lin_\field \A^d$.
        As $\A$ is oligomorphic, the basis $\A^d$ is orbit-finite; 
        in particular $\Lin_\field \A^d$ is orbit-finitely spanned.
        But so is $U$ by~(\ref{item:sub-ofs-is-ofs}), as is the quotient $U/W$. 
        \item 
        In the other direction, take some orbit-finitely spanned vector space $V$. 
        By Definition~\ref{def:orbit-finite-set}, the spanning set can be obtained from some equivariant subset $X$ of $\A^d$ by quotienting under an equivariant equivalence relation. 
        This gives us a surjective equivariant linear map from the equivariant subspace $U = \Lin_\field X \subseteq \Lin_\field \A^d$ to $V$; 
        call its kernel $W$, so that we have a bijective equivariant linear map $U/W \to V$. 
    \end{itemize}

    Finally the implication (\ref{item:ofs-implies-Noetherian}) $\Rightarrow$ (\ref{item:LinAd-is-Noetherian}) is immediate, 
    whilst the converse implication follows from the ``furthermore'' part: 
    if there were an infinite chain $V_0 \subset V_1 \subset \cdots$ in $U/W$,
    putting \[
        U_i = \{u \in U \mid\text{ $u+W$ is in $V_i$ }\}
    \] would give rise to an infinite chain $U_0 \subset U_1 \subset \cdots$ in $U$ and also in the bigger $\Lin_\field \A^d$.
\qed

%% file: src/appendix-symplectic.tex
\section{Proof of Theorem~\ref{thm:have-oligomorphic-approximation}(\ref{item:rado-has-oligomorphic-approximation})}\label{sec:appendix-symplectic}
To show that the Rado graph is oligomorphically approximated, we introduce a family of graphs defined from symplectic vector spaces over the two-element field.
These spaces, described as a classical geometry, are treated in classical textbooks such as \cite[pp.~75--83]{Aschbacher_symplectic} and used as a basic building block in~\cite{KLM89}.
Nonetheless, we still find it useful to present a self-contained exposition here, aimed at a broader audience.
We will first recall the proof that the countable-dimensional symplectic vector space is smoothly approximated,
and then explain how this helps with the Rado graph.

\subsection{Symplectic vector spaces}
Fix a finite field $\ff$.
\begin{definition}
    A \emph{symplectic vector space} is a $\ff$-vector space $\W$ 
    equipped with a bilinear form $\omega: \W \times \W \to \ff$ that is
    \begin{bracketenumerate}
        \item \emph{alternating}: $\omega(v, v) = 0$ for all $v$; and
        \item \emph{non-degenerate}: if $\omega(v, w) = 0$ for all $w$, then $v = 0$.
    \end{bracketenumerate}    
    Model-theoretically, we start with the vocabulary of the vector atoms from Example~\ref{ex:vector-space-atoms} 
    and, to speak about $\omega$, add a binary relation $\omega(-, -) = \lambda$ for every $\lambda \in \ff$.
\end{definition}

\begin{example}\label{ex:canonical-symplectic-space}
    Let $\W_n$ be the $\ff$-vector space with basis $e_1  , \ldots, e_n, f_1, \ldots, f_n$.
    Define $\omega$ by bilinearly extending
    \begin{equation}\label{eq:symplectic-basis}
        \omega(e_i, f_i) = 1 = -\omega(f_i, e_i),\quad
        \omega(-, *) = 0 \text{ elsewhere;}
        \tag{$\ddagger$}
    \end{equation}
    one may straightforwardly check that $\omega$ is alternating and non-degenerate.
    Moreover, as $\W_0 \subseteq \W_1 \subseteq \W_2 \subseteq \cdots$,
    we obtain a countable-dimensional symplectic vector space $\bigcup_n \W_n$.
\lipicsEnd\end{example}

We will refer to any vectors $e_1, \dots, e_n, f_1, \dots, f_n$ satisfying~\eqref{eq:symplectic-basis} as a \emph{symplectic subbasis}. 
(This is not standard terminology.)
As the name suggests, such vectors must be linearly independent:
if \(
    v = \sum_i \lambda_i e_i + \mu_i f_i
\) is the zero vector, then $\lambda_i = \omega(v, f_i) = 0$ and $\mu_i = \omega(e_i, v) = 0$ for each $i$.
We shall see that symplectic subbases behave very much like linearly independent subsets.

\begin{lemma}\label{lem:symplectic-basis}
    Assume that $\W$ is a symplectic vector space that is at most countable.
    Then any finite symplectic subbasis $e_1, \dots, e_n, f_1, \dots, f_n$
    can be extended to a symplectic subbasis that spans the whole $\W$ (i.e.~a symplectic basis).
\end{lemma}
\begin{proof}
    Suppose that $e_1, \ldots, e_n, f_1, \ldots, f_n$ do not already span $\W$;
    take $v$ to be a witness (that is least with respect to some fixed enumeration of $\W$ 
    in the case it is infinite).
    Put
    \[
        e_{n+1} = v - \sum_{i=1}^n \omega(e_i, v) f_i + \sum_{i=1}^n \omega(f_i, v) e_i
    \]
    so that $\omega(e_i, e_{n+1}) = 0 = \omega(f_i, e_{n+1})$ for $1 \leq i \leq n$.
    This cannot be the zero vector lest we contradict the choice of $v$.
    As $\omega$ is non-degenerate, there is~--~rescaling if necessary~--~some $w$ such that $\omega(e_{n+1}, w) = 1$. 
    Now define
    \[
        f_{n+1} = w - \sum_{i=1}^n \omega(e_i, w) f_i + \sum_{i=1}^n \omega(f_i, w) e_i
    \]
    in a similar manner, 
    making $e_1, \ldots, e_n, e_{n+1}, f_1, \ldots, f_n, f_{n+1}$ a symplectic subbasis whose span contains $v$.
    We go through every element of $\W$ by continuing this way.
\end{proof}

We may even allow some $e_i$ to appear without a corresponding $f_i$, under a finiteness assumption:
\begin{lemma}\label{lem:symplectic-basis-and-a-half}
    Now assume $\W$ is a finite-dimensional symplectic vector space.
    Let 
    \begin{align*} 
        &e_1, \ldots, e_n, e_{n+1}, \ldots, e_{n+k}, \\
        &f_1, \ldots, f_n
    \end{align*}
    be linearly independent vectors satisfying \eqref{eq:symplectic-basis}.
    Then we can extend this collection to a symplectic basis of $\W$.
\end{lemma}
\begin{proof}
    Suppose we have found $f_{n+1}, \ldots, f_{n+i}$ already such that
    \begin{align*}
        &e_1, \ldots, e_n, e_{n+1}, \ldots, e_{n+i}, e_{n+i+1}, e_{n+i+2}, \ldots, e_{n+k},\\ 
        &f_1, \ldots, f_n, f_{n+1}, \ldots, f_{n+i}
    \end{align*}
    satisfy \eqref{eq:symplectic-basis}.
    Notice that these vectors are linearly independent:
    in a linear combination that sums to $0$,
    the coefficients of $e_1, \dots, e_{n+i}, f_1, \dots, f_{n+i}$ must be zero;
    but we assumed that $e_{n+i+1}, \dots, e_{n+k}$ are linearly independent.
    By extending these linearly independent vectors to a basis $B$ of $\W$, we may define a linear function 
    \(
        \psi : \W \to \ff
    \)
    that sends $e_{n+i+1}$ to $1$ but every other $b \in B$ to $0$.
    Now apply Lemma~\ref{lem:symplectic-basis} to obtain a new symplectic basis $e'_1, \dots, e'_m, f'_1, \dots, f'_m$ of $\W$, and put
    \[
        f_{n+i+1} = \sum_{j=1}^m \psi(e'_j) f'_j - \psi(f'_j) e'_j.
    \]
    Then $\omega(-, f_{n+i+1})$ agrees with $\psi$ on this symplectic basis, so by linearity they must be the same function.
    In particular 
    \(
        \omega(e_{n+i+1}, f_{n+i+1}) = \psi(e_{n+i+1}) = 1,
    \)
    whereas $\psi(e_1), \dots, \psi(e_{n+k}), \psi(f_1), \dots, \psi(f_{n+i})$ are all $0$.
    Thus we have \eqref{eq:symplectic-basis} as required.
    Finally, to complete this symplectic subbasis to a symplectic basis, we use Lemma~\ref{lem:symplectic-basis} again.
\end{proof}

Given two symplectic vector spaces $\W$ and $\W'$,
note that an embedding $\alpha$ from $X \subseteq \W$ to $X' \subseteq \W'$ is an injective linear map (from the span of $X$ to the span of $X'$) that satisfies $\omega(x_1, x_2) = \omega(\alpha(x_1), \alpha(x_2))$ for all $x_1, x_2 \in X$.
We call any map satisfying the latter condition \emph{isometric}. 
(This is more standard terminology.)

Now, if $e_1, \dots, e_n, f_1, \dots, f_n$ is a symplectic basis for $\W$ and $e'_1, \dots, e'_n, f'_1, \dots, f'_n$ is a symplectic basis for $\W'$ of the same size,
then $e_i \mapsto e'_i, f_i \mapsto f'_i$ defines an isometric linear bijection $\W \to \W'$, i.e.~an isomorphism.
It follows by Lemma~\ref{lem:symplectic-basis} that, up to isomorphism, the $\W_n$'s from Example~\ref{ex:canonical-symplectic-space} are all the finite-dimensional symplectic $\ff$-vector spaces;
similarly, $\bigcup_n \W_n$ is the unique countable(-dimensional) symplectic $\ff$-vector space.
We shall now show that they are all homogeneous:

\begin{lemma}[Witt's Extension Theorem] \label{lem:symplectic-witt-extension}
    For $n = 0,1,\dots$, any embedding $\alpha$ of $X \subseteq \W_n$ into $\W_n$ 
    can be extended to an automorphism of $\W_n$.
\end{lemma}
\begin{proof}
    By taking the span of $X$ and extending $\alpha$ linearly, we may assume that $X$ is a subspace of $\W_n$.
    Consider the subspace
    \[
        X_0 = \{ x \in X \mid\text{ $\omega(x, x') = 0$ for all $x' \in X$ }\} 
    \]
    of $X$, and find a complementary subspace $X_1$ so that $X = X_0 \oplus X_1$.  
    Notice that $\omega$ is still non-degenerate (and certainly alternating) when restricted to $X_1$:
    as $\omega(x_1, x'_0 + x'_1) = \omega(x_1, x'_1)$, if $x_1 \in X_1$ satisfies $\omega(x_1, x'_1) = 0$ for all $x'_1 \in X_1$, then $x_1$ must be in $X_0 \cap X_1 = \{0\}$.
    So $X_1$ is a symplectic subspace, and we may apply Lemma~\ref{lem:symplectic-basis} to find a symplectic basis $e_1, \dots, e_d, f_1, \dots, f_d$ for $X_1$. 
    
    Now let $x_1, \dots, x_k$ be any basis for $X_0$.
    Then the vectors
    \begin{align*}
        &e_1, \ldots, e_d, x_1, \ldots, x_k, \\
        &f_1, \ldots, f_d
    \end{align*}
    form a basis for $X$ and satisfy \eqref{eq:symplectic-basis}.
    As $\alpha$ is an isometric linear injection, the vectors
    \begin{align*}
        &\alpha(e_1), \ldots, \alpha(e_d), \alpha(x_1), \ldots, \alpha(x_k), \\
        &\alpha(f_1), \ldots, \alpha(f_d)
    \end{align*}
    form a basis for $\alpha(X)$ and also satisfy \eqref{eq:symplectic-basis}.
    Therefore, we may apply Lemma~\ref{lem:symplectic-basis-and-a-half} twice and extend these to two symplectic bases for $\W_n$.
    The two bases must have the same size, and we thus obtain an automorphism of $\W_n$ extending $\alpha$.
\end{proof}

\begin{corollary}\label{cor:W_infty-is-homogeneous}
    $\bigcup_n \W_n$ is also homogeneous.
\end{corollary}
\begin{proof}
    Let $\alpha$ be an embedding of a finite subset $X \subseteq \bigcup_n \W_n$ into $\bigcup_n \W_n$.
    Then $\alpha$ is an embedding of $X \subseteq \W_{n'}$ into $\W_{n'}$ for some large enough $n'$,
    so by Lemma~\ref{lem:symplectic-witt-extension} some $\beta \in \Aut(\W_{n'})$ restricts to $\alpha$.
    But we can easily extend $\beta$ to an automorphism of $\bigcup_n \W_n$ by mapping $e_{n'+i} \mapsto e_{n'+i}, f_{n'+i} \mapsto f_{n'+i}$ for all $i \geq 1$.
\end{proof}

To prove that $\bigcup_n \W_n$ is smoothly approximated by the family of its finite-dimensional symplectic subspaces, 
we still need to show that $\bigcup_n \W_n$ is oligomorphic.
Whilst this follows from $\omega$-categoricity, we will use homogeneity to find an exact orbit count:

\begin{lemma}\label{lem:symplectic-oligomorphic}
    Write $q = \left\vert\ff\right\vert$.
    In the countable $\ff$-vector space and the countable symplectic $\ff$-vector space respectively, 
    there are
    \(
        \sum_{m = 0}^d \begin{bsmallmatrix}d \\ m \end{bsmallmatrix}_q
    \) 
    and
    \(
        \sum_{m = 0}^d \begin{bsmallmatrix}d \\ m \end{bsmallmatrix}_q \cdot q^{\binom{m}{2}}
    \) orbits of $d$-tuples.
    Here we use the Gaussian q-binomial coefficient \[
        \begin{bmatrix}d \\ m \end{bmatrix}_q 
        = \frac{
            (1 + q + \cdots + q^{d-1}) \cdot (1 + q + \cdots + q^{d-2}) \cdots (1 + q + \cdots + q^{d-m})
        }{
            (1 + q + \cdots + q^{m-1}) \cdot (1 + q + \cdots + q^{m-2}) \cdots (1 + q) \cdot 1
        }.
    \] 
\end{lemma}
\begin{proof}
    To each $d$-tuple $v_\bullet \in (\bigcup_n \W_n)^d$ we associate the following data:
    \begin{bracketenumerate}
        \item 
        pivot indices $I \subseteq \{1, \ldots, d\}$,
        where $i \in I$ if and only if $v_i$ is not spanned by the vectors $v_1, \ldots, v_{i-1}$ that come before;

        \item 
        an assignment $\lambda_j : \{i \in I \mid i < j\} \to \ff$, for each $j \not\in I$,
        such that $v_j = \sum_{i \in I, i < j} \lambda_j(i) \cdot v_i$;

        \item 
        and a map $\mu : \binom{I}{2} \to \ff$, where $\mu(\{i > i'\}) = \omega(v_{i}, v_{i'})$ for pivot indices $i, i' \in I$.
    \end{bracketenumerate}
    Then two $d$-tuples are:
    \begin{itemize}
        \item 
        in the same orbit under linear bijections $\bigcup_n \W_n \to \bigcup_n \W_n$ if and only if they have the same data~(1) and~(2), 
        since $\bigcup_n \W_n$ is homogeneous as a vector space (see Example~\ref{ex:vector-space-atoms}); and

        \item
        in the same orbit under isometric linear bijections $\bigcup_n \W_n \to \bigcup_n \W_n$ if and only if they have the same data~(1),~(2), and~(3),
        since $\bigcup_n \W_n$ is moreover homogeneous as a symplectic vector space by Corollary~\ref{cor:W_infty-is-homogeneous}.
    \end{itemize} 
    Also, given any $(I, \{\lambda_j\}_{j \not\in I}, \mu)$, we can find some $v_\bullet \in (\bigcup_n \W_n)^d$ with these data~--~put:
    \[
        v_i = \sum_{i' \in I, i' < i} \mu(\{i > i'\}) \cdot e_{i'} + f_i\text{ for $i \in I$,}\quad
        \text{and }v_j = \sum_{i \in I, i < j} \lambda_j(i) \cdot v_i \text{ for $j \not\in I$}.
    \]
    So it remains to check that, for $m \leq d$, there are $\begin{bsmallmatrix}d \\ m \end{bsmallmatrix}_{\left\vert\ff\right\vert}$ choices for $(I, \{\lambda_j\}_{j \not\in I})$ with $I$ being an $m$-element subset of $\{1, \dots, d\}$.

    We proceed by induction on $d$, noting that the counts match if $m = 0$:
    we have $\begin{bsmallmatrix}d \\ 0\end{bsmallmatrix}_{\left\vert\ff\right\vert} = 1$.
    Now if $I$ is an $(m+1)$-element subset of $\{1, \dots, d, d + 1\}$, 
    either $d+1 \in I$ or $I \subseteq \{1, \dots, d\}$.
    \begin{itemize}
        \item 
        In the first case, $I' = I \setminus \{d+1\}$ is an $m$-element subset of $\{1, \dots, d\}$, 
        and by the inductive hypothesis we have $\begin{bsmallmatrix}d \\ m\end{bsmallmatrix}_{\left\vert\ff\right\vert}$ choices for the $\lambda_j$'s with $j \in \{1, \dots, d\} \setminus I' = \{1, \dots, d, d+1\} \setminus I$.

        \item 
        In the second case, by the inductive hypothesis we have $\begin{bsmallmatrix}d \\ m+1\end{bsmallmatrix}_{\left\vert\ff\right\vert}$ choices for the $\lambda_j$'s with $j \in \{1, \dots, d\} \setminus I$, 
        and we have $\left\vert \ff^{m+1} \right\vert$ choices for $\lambda_{d+1}$~--~these can be made independently.
    \end{itemize}
    Altogether, there are
    \[
        \begin{bmatrix}d \\ m\end{bmatrix}_{\left\vert \ff \right\vert} + \begin{bmatrix}d \\ m+1\end{bmatrix}_{\left\vert \ff \right\vert} \cdot \left\vert \ff \right\vert^{m+1}
        = \begin{bmatrix}d+1 \\ m+1\end{bmatrix}_{\left\vert \ff \right\vert}
    \]
    choices for $(I, \{\lambda_j\}_{j \in \{1, \dots, d+1\} \setminus I})$.
\end{proof}

In light of the results above, we conclude:
\begin{corollary}
    For any finite field $\ff$, the countable symplectic $\ff$-vector space $\bigcup_n \W_n$ is smoothly approximated by all its finite-dimensional symplectic subspaces.
    By Theorems~\ref{thm:have-oligomorphic-approximation}(\ref{item:smooth-implies-oligomorphic-approximation}) and~\ref{thm:weak-smooth-approximation-finite-length}, 
    it has the finite length property over any field of characteristic $0$ (unlike $\ff$ itself).
\end{corollary}

\subsection{Symplectic graphs}
Let us now explain the connection with the Rado graph and finish the proof of Theorem~\ref{thm:have-oligomorphic-approximation}(\ref{item:rado-has-oligomorphic-approximation}).
Here, let $\ff = \ff_2$ be the finite field with two elements.
\begin{definition}
    For $n = 0, 1, 2, \ldots$, the \emph{symplectic graph} $\overline{\W_n}$ has vertices $\W_n$ and edges
    \[
        E(v_1, v_2) \;:\Longleftrightarrow\; \omega(v_1, v_2) = 1.
    \]
    This is indeed an undirected graph: as $\omega$ is alternating, we have $\omega(v_1, v_2) = -\omega(v_2, v_1) = \omega(v_2, v_1)$ over $\ff_2$.
\end{definition}

\begin{lemma}\label{lem:symplectic-vs-graph}
    $\Aut(\overline{\W_n}) = \Aut(\W_n)$.
\end{lemma}
\begin{proof}
    Clearly each isometric linear bijection $\W_n \to \W_n$ is a graph automorphism of $\overline{\W_n}$.
    Conversely, every graph automorphism $f \in \Aut(\overline{\W_n})$ is evidently isometric.
    To show that $f$ is linear, take $\lambda_1, \lambda_2 \in \ff$ and $v_1, v_2 \in \W_n$.
    We calculate:
    \begin{align*}
        &\omega\Bigl( f(\sum_i \lambda_i v_i) - \sum_i \lambda_i f(v_i) ,\: f(w) \Bigr) 
        = \omega\Bigl( f(\sum_i \lambda_i v_i), f(w)\Bigr) - \sum_i \lambda_i \omega\bigl(f(v_i), f(w) \bigr) \\
        ={}& \omega\Bigl( \sum_i \lambda_i v_i, w \Bigr) - \sum_i \lambda_i \omega( v_i, w ) 
        = \omega(0, w) = 0
    \end{align*}
    for all $f(w) \in f(\W_n) = \W_n$;
    since $\omega$ is non-degenerate, 
    we conclude that $f(\sum_i \lambda_i v_i) = \sum_i \lambda_i f(v_i)$.
\end{proof}

So the number of orbits in $\overline{\W_n}^d$ is equal to the number of orbits in $(\W_n)^d$,
which by Lemma~\ref{lem:symplectic-witt-extension} and Corollary~\ref{cor:W_infty-is-homogeneous} is bounded from above by the number of orbits in $(\bigcup_n \W_n)^d$.
(We have equality if $n \geq d$.)
This last depends on $d$ and not on $n$.

\begin{remark}
    Let $\left\{\begin{smallmatrix}d \\ m\end{smallmatrix}\right\}$ count the number of equivalence relations with $m$ classes on a $d$-element set.
    We have the following orbit counts of $d$-tuples:\footnote{These are the $d$-th terms in the OEIS sequences \href{https://oeis.org/A335390}{A335390}, \href{https://oeis.org/A028361}{A028361}, and \href{https://oeis.org/A006116}{A006116} respectively.}
    \[
        \myunderbrace{\sum_{m=0}^d \left\{\begin{matrix}d \\ m \end{matrix}\right\} \cdot 2^{\binom{m}{2}}}{Rado graph}
        \quad\leq\quad \myunderbrace{\sum_{m = 0}^d \begin{bmatrix}d \\ m \end{bmatrix}_2 \cdot 2^{\binom{m}{2}}}{$\bigcup_n \W_n$}
        \quad\geq \myunderbrace{\sum_{m = 0}^d \begin{bmatrix}d \\ m \end{bmatrix}_2.}{$\ff_2$-vector space}
    \]
\end{remark}

So we only need to show that any finite graph $G$ can be embedded into some $\overline{\W_{n(G)}}$:

\begin{lemma}[{\cite[Thm.~8.11.2]{GR01}}]
    Every graph on at most $2n$ vertices embeds in $\overline{\W_n}$.
\end{lemma}
\begin{proof}[Proof (without jargon)]

    Let $G$ be a graph on at most $2n$ vertices. 
    Note that if $G$ contains no edges, we can choose any $2n$ of the $2^n$ vectors in 
    the span of $e_1, \ldots, e_n$ in $\overline{\W_n}$.
    
    So suppose that $n \geq 1$ and $G$ has an edge $E(s, t)$.
    Let $G_{s,t}$ be the graph with vertices $G \setminus \{s, t\}$ and with edges $E_{s,t}$ that we will specify later.
    By induction, some embedding $f : G_{s, t} \to \overline{\W_{n-1}}$ exists.
    Define $f' : G \to \overline{\W_n}$ by
    \begin{align*}
        x \in G_{s, t} \mapsto f(x) - \lceil E(x, s) \rfloor f_n + \lceil E(x, t) \rfloor  e_n,\quad 
        s \mapsto e_n,\quad 
        t \mapsto f_n
    \end{align*}
    where $\lceil \Phi \rfloor $ is $1$ if $\Phi$ holds, and $0$ otherwise.
    On one hand, we have
    \begin{align*}
        \omega(f'(x), f'(s)) = \lceil E(x, s) \rfloor, \quad
        \omega(f'(x), f'(t)) = \lceil E(x, t) \rfloor 
    \end{align*}
    as desired.
    On the other,
    \begin{align*}
        \omega( f'(x_1), f'(x_2) ) 
        = \lceil E_{s,t}(x_1, x_2) \rfloor  
        + \lceil E(x_1, s) \rfloor  \lceil E(x_2, t) \rfloor
        - \lceil E(x_1, t) \rfloor  \lceil E(x_2, s) \rfloor 
    \end{align*}
    tells us how we should define the edge relation $E_{s,t}$ in $G_{s,t}$ for $f'$ to be an embedding of graphs, 
    i.e.~so that $\omega( f'(x_1), f'(x_2) ) = \lceil E(x_1, x_2) \rfloor$.
\end{proof}

Note that $\bigcup_{n} \overline{\W_n}$ is not isomorphic to the Rado graph.
For instance, there are no common neighbours of $e_i$, $f_i$, and $e_i + f_i$:
to have $1 = \omega(v, e_i + f_i) = \omega(v, e_i) + \omega(v, f_i) = 1 + 1$ is impossible.
As both $\bigcup_{n} \overline{\W_n}$ and the Rado graph embed all finite graphs, the homogeneity of the latter means that $\bigcup_{n} \overline{\W_n}$ cannot be a homogeneous graph.
Also, the zero vector is an isolated vertex, so $\overline{\W_1}, \overline{\W_2}, \dots$ are not homogeneous either~--~so this family is not a smooth approximation, but only an oligomorphic approximation of the Rado graph:

\begin{corollary}
    The Rado graph is oligomorphically approximated by the finite symplectic graphs (embedded in all possible ways in the Rado graph).
    By Theorem~\ref{thm:weak-smooth-approximation-finite-length}, it has the finite length property over any field of characteristic $0$.
\end{corollary}

A different finite field $\innerfield$ can encode different structures.
For example, the 3-element field allows us to analogously prove the finite length property of the universal homogeneous digraph over any field of characteristic zero.\footnote{We are again grateful to Ehud Hrushovski for pointing this out.}

%% file: src/appendix-function-spaces.tex
\section{Example~\ref{ex:no-function-spaces} continued}\label{sec:appendix-no-function-spaces}
In this part of the appendix, we prove that the space $V$ from Example~\ref{ex:no-function-spaces} is not orbit-finitely spanned.
Recall that $f_S$ is, given a finite subset $S$ of the Rado atoms $\A$, the characteristic function of the common neighbours of $S$ in $\A$;
the span of all such $f_S$ is $V$.

Suppose towards a contradiction that $V$ has an orbit-finite spanning set. 
Without loss of generality we can assume that this spanning set uses only vectors of the form $f_S$, 
i.e.~the spanning set is $\{ f_S \mid S \in \mathscr{S} \}$ for some orbit-finite family $\mathscr{S}$ of finite subsets of $\A$.
Take any finite set $T \subset \A$ that has strictly more elements than any set in $\mathscr{S}$ 
(which is possible because subsets of $\A$ in the same orbit, of the finitely many orbits, have the same number of elements). 
Then we can write
\[
    f_T = \lambda_1 \cdot f_{S_1} + \cdots + \lambda_n \cdot f_{S_n}
\]
for some distinct $S_i \in \mathscr{S}$.  
We will prove that all coefficients $\lambda_i$ are zero, and hence $f_T$ is the zero function. 
This cannot be true, since one can find an atom that is a common neighbour of all the atoms in $T$.

To prove that $\lambda_i = 0$, we use induction on $n$. 
In the induction proof we assume without loss of generality that the sets $S_1,\ldots,S_n$ are sorted by size in a non-decreasing way. 
Suppose that we have proved that all coefficients $\lambda_j$ are zero for $j < i$. 
By the assumption on non-decreasing sizes, each of the sets $S_{i+1},\ldots,S_{n}$ and $T$ contains at least one atom that is not in $S_i$. 
Therefore, we can find some $a \in \A$ that is adjacent to all atoms in $S_i$, 
but which is non-adjacent to at least one atom in each of the sets $S_{i+1},\ldots,S_{n}$ and $T$. 
As $\lambda_1, \dots, \lambda_{i-1} = 0$, we see that
\[
    0 = f_T(a) = \sum_{j \geq i} \lambda_j \cdot f_{S_j}(a) = \lambda_i \cdot 1 + \sum_{j > i} \lambda_j \cdot 0
\]
and hence $\lambda_i=0$.

%% file: src/appendix-woodrow-lachlan.tex
\section{Proof of Corollary~\ref{cor:lachlan-woodrow}}\label{sec:appendix-lachlan-woodrow}
Let us begin by listing a few straightforward applications of Theorem~\ref{thm:ordered-free-amalg-has-finite-length}:
the generically ordered expansion $\A$ of each $\A_0$ in Figure~\ref{fig:some-free-amalg-classes} has the finite length property over any field.
Looking at rows~1 and~4 there, by Corollary~\ref{cor:free-finite-length} we immediately obtain Corollary~\ref{cor:lachlan-woodrow}(i) and~(iv):
\begin{claim}\label{claim:ordered-and-equality-have-finite-length}
    In light of Example~\ref{ex:generically-ordered-equality}, the ordered atoms $(\Q, <)$ and the equality atoms $(\N, =) \cong (\Q, =)$ have the finite length property over any field.
\end{claim}

\begin{claim}
    Given any set $\mathscr{T}$ of finite tournaments, Henson's $\mathscr{T}$-free digraph has the finite length property over any field.
\end{claim}

\begin{figure}
\begin{tabular}{llll}
   & $\sigma_0$ & $\Cclass_0$ & $\A_0$ \\
    \toprule
    1.& $\{\ \}$ & $\Forb(\{\})$ & $(\N, =) \:\cong\: (\Q, =)$ \\
    2.& $\{\ E(-,-)\ \}$ & $\Forb(\{\bigdot \rotatebox{90}{$\circlearrowleft$}, {\bigdot} {\rightleftarrows} {\bigdot}\})$ & Rado graph \\
    3.& $\{\ E(-,-)\ \}$ & $\Forb(\{\bigdot \rotatebox{90}{$\circlearrowleft$}, {\bigdot} {\rightleftarrows} {\bigdot}, K_n\})$, & Henson's $K_n$-free graph \\
    & & \quad for any $n \geq 3$ & \\
    4.& $\{\ E(-,-)\ \}$ & $\Forb(\{\bigdot \rotatebox{90}{$\circlearrowleft$}, {\bigdot} {\rightarrow} {\bigdot}\} \cup \mathscr{T})$, & Henson's $\mathscr{T}$-free digraph \\
    & & \quad for any set $\mathscr{T}$ of finite tournaments & \\
    5.& $\{\ P(-)\ \}$ & $\Forb(\{\})$ & $(\N, \text{even numbers})$ \\
    & & & $\cong (\{e^{i r} \mid r \in \Q\}, \text{right half-circle})$ \\
    \bottomrule
\end{tabular}
\caption{Some Fraïssé limits of free amalgamation classes over at-most-binary vocabularies}\label{fig:some-free-amalg-classes}
\end{figure}

To deduce the finite length property of more structures, we now present a more general version of Corollary~\ref{cor:free-finite-length}.
For that, we need the notion of first-order interpretation.
\begin{definition}
    Let $\A$ be an oligomorphic $\sigma$-structure.
    A $\sigma'$-structure $\B$, where $\sigma'$ is a relational vocabulary possibly different from $\sigma$,
    is \emph{first-order interpretable} in $\A$ if there exist:
    \begin{bracketenumerate}
        \item an orbit-finite set $X$ over $\A$, and
        \item\label{item:fo-interpretation-definable-relations} for each $R \in \sigma'$, with $R$ being $d$-ary, an $\Aut(\A)$-equivariant subset $X_R \subseteq X^{d}$
    \end{bracketenumerate}
    such that $\B$ and $(X, \{X_R\}_{R \in \sigma'})$ are isomorphic as $\sigma'$-structures.
\end{definition}

In particular, note that any first-order reduct $\B$ of $\A$ is first-order interpretable in $\A$ by using $\A$ itself as the underlying orbit-finite set.
In that case, as $\A$ is oligomorphic, we may replace ``$\Aut(\A)$-equivariant'' in (\ref{item:fo-interpretation-definable-relations}) with ``$\sigma$-definable''.

\begin{lemma}\label{lem:interpretation-preserves-finite-length}
    If an oligomorphic structure $\A$ has the finite length property over a given field $\field$,
    then so does any structure $\B$ that is first-order interpretable in $\A$.
\end{lemma}
\begin{proof}
    Identify $\B$ with an orbit-finite set over $\A$ equipped with $\Aut(\A)$-equivariant relations.
    Then, given $\pi \in \Aut(\A)$, the induced action $b \mapsto \pi(b)$ on $\B$ is an automorphism of $\B$.
    So an $\Aut(\B)$-orbit of tuples in any $\B^d$ is also equivariant under $\Aut(\A)$, i.e.~it is a union of $\Aut(\A)$-orbits.
    Since $\B^d$ is orbit-finite over $\A$, it can only have finitely many $\Aut(\B)$-orbits.
    In short, $\B$ is oligomorphic because $\A$ is so.

    Similarly, consider a chain of $\Aut(\B)$-equivariant subspaces in $\Lin_\FF \B^d$.
    These spaces are necessarily $\Aut(\A)$-equivariant, 
    so the chain must have bounded length by the finite length property of $\A$~--~since $\B^d$ is an orbit-finite set over $\A$.
\end{proof}

In addition to the above lemma, to prove Corollary~\ref{cor:lachlan-woodrow}(ii) and (iii) we need to rely on two classification results due to Lachlan and Woodrow. 
These have intricate proofs but are simple to state:

\begin{fact}[{\cite{LachlanWoodrow_80}}]
    Any countable homogeneous graph is, up to complementation, isomorphic to one of: 
    \begin{romanenumerate}
        \item the Rado graph;
        \item Henson's $K_n$-free graph, where $n \geq 3$;
        \item $m$ disjoint copies of the complete graph $K_n$, where at least one of $m$ and $n$ is infinite.
    \end{romanenumerate}    
\end{fact}

\begin{claim}
    All countable homogeneous graphs have the finite length property over any field.
\end{claim}
\begin{claimproof}
    As the complement of a graph $\A$ is a first-order reduct of $\A$,
    when studying the length of $\Lin_\FF \A^d$ we only need to examine the structures of kinds (i), (ii) and (iii). 
    For the first two kinds, the finite length property follows directly from Corollary~\ref{cor:free-finite-length} by looking at rows~2 and~3 of Figure~\ref{fig:some-free-amalg-classes}. 
    The structures of kind (iii) are all first-order interpretable in the ordered atoms~--~see~\cite[Prop.~2.19 and Cor.~2.20]{BodirskyBodorMarimon_25}, 
    and hence inherit the finite length property from them by Lemma~\ref{lem:interpretation-preserves-finite-length}.
\end{claimproof}

\begin{fact}[{\cite{Lachlan1984}}]
    Any countable homogeneous tournament is isomorphic to one of:
    \begin{romanenumerate}
        \item $(\Q, <)$;
        \item the Fraïssé limit of all finite tournaments;
        \item 
        the \emph{dense local order} $(\{e^{i r} \mid r \in \Q\}, \to)$ where $x \to y$ if and only if $0 < \arg(y / x) < \pi$,
        i.e.~from $x$ to $y$ (which cannot be antipodal) the shorter circular path is anticlockwise.
    \end{romanenumerate}        
\end{fact}

\begin{claim}
    All countable homogeneous tournaments have the finite length property over any field.
\end{claim}
\begin{claimproof}
    \begin{romanenumerate}
        \item
        Already shown in Claim~\ref{claim:ordered-and-equality-have-finite-length}.

        \item
        Recall that the ordered Rado graph, i.e.~the generically ordered expansion $(\A, E, <)$ of $(\A_0, E)$ from row~2 in Figure~\ref{fig:some-free-amalg-classes}, has the finite length property over any field.
        Consider its first-order reduct $(\A, \to)$, where we define $x \to y$ by
        \[
            x \neq y \qquad\wedge\qquad E(x, y) \iff x < y.
        \]
        Then $(\A, \to)$ is a tournament that embeds all finite tournaments in such a way that, by a back-and-forth argument, one sees that it is homogeneous and isomorphic to the Fraïssé limit of all finite tournaments \cite[Sec.~2.2.1]{BPP15}.
        Consequently, this latter structure inherits the finite length property from the ordered Rado graph.

        \item
        This time, consider $\mathbb{S} = \{ e^{i r} \mid r \in \Q \}$ equipped with $\to$ from (iii) and with
        \[
            P(z) \iff -\frac{\pi}{2} < \arg(z) < \frac{\pi}{2}
        \]
        from row~5 in Figure~\ref{fig:some-free-amalg-classes}.
        Define a new binary relation $x < y$ by
        \[
            x \neq y \qquad\wedge\qquad (P(x) \Leftrightarrow P(y)) \iff x \to y.
        \]
        Then $(\mathbb{S}, P, <)$ embeds all total orderings of all $\{P\}$-structures and satisfies an extension property that, 
        via a back-and-forth argument, shows it is homogeneous and isomorphic to the generically ordered expansion of $(\mathbb{S}, P)$~--~see~\cite[p.~18]{Nvt13}. 
        Hence $(\mathbb{S}, P, <)$ has the finite length property over any field by Theorem~\ref{thm:ordered-free-amalg-has-finite-length}. 
        But we may define $\to$ by an analogous formula using $P$ and $<$, so the dense local order $(\mathbb{S}, \to)$ inherits this finite length property as a first-order reduct.
        \claimqedhere
    \end{romanenumerate}
\end{claimproof}

%% file: src/appendix-cogs.tex
\section{Proof of Theorem~\ref{thm:cog-span-generally}}\label{sec:appendix-cogs}
Recall the setting of Section~\ref{sec:free-amalg} and the notations of Section~\ref{sec:cogs-turn}.
Consider a subspace $\EE \subseteq \FF^n$, and heed the warning of~\ref{rem:FF-EE}:
remember that a cog is of the form $\lambda \cdot a \between b$ now.
Finally, fix an $S$-ordered orbit $\cal O \subseteq \A^I$.

\subsection{Subvectors, locations, conflicts}
We begin by introducing some additional terminology and notation. 
First, let us make explicit a view we have tacitly taken:
with $\cal O$ as a standard basis, a vector $v \in \Lin_\EE \cal O$ is just a finite set of pairs in $\EE \times \cal O$.
A \emph{subvector} of $v$ is a subset of these pairs. 
We write $\vsup{v}\subseteq {\cal O}$ for the set of tuples that are present in $v$. 
For a finite subset $T\subseteq\cal O$, we write $\sqrt T\subseteq\A$ for the set of atoms present anywhere in $T$.

For any $i \in I$ and $a \in \A$ (which is equal to $b_i$ for some $b \in \cal O$), we write 
\[
    \cal O^{i:a} = \{c \in \cal O \mid c_i = a\};
\]
this is an $\Aut(\A / S a)$-orbit (containing $b$), and its projection $\cal O^{i:a} |^{-i}$ is $Sa$-ordered.
For a vector $v \in \Lin_\EE \cal O$, by
\[
    v^{i:a} \in \Lin_\EE \cal O^{i:a}
\]
we mean the subvector of $v$ consisting of all pairs in $\EE \times \cal O^{i:a}$.

\begin{lemma}\label{lem:balanced-projected-subvector}
    Let $v \in \Lin_\EE \cal O$ be balanced. 
    Then any projected subvector $v^{i:a}|^{-i} \in \Lin_\EE \cal O^{i:a}|^{-i}$ is also balanced.
\end{lemma}
\begin{proof}
    Let $j \in I \setminus \{i\}$. 
    By assumption we have \[
        0 = v|^{-j} = \sum_a v^{i:a}|^{-j}.
    \]
    This sum is finite: it runs over those atoms $a$ that occur as the $i$-th entries in $\vsup{v}$.
    By looking at $i$-th entries, we see that each $v^{i:a}|^{-j}$ must be the zero vector.
    Hence $0 = v^{i:a}|^{-j}|^{-i} = v^{i:a}|^{-i}|^{-j}$,
    which shows that $v^{i:a}|^{-i}$ is balanced.
\end{proof}

For a finite subset $T \subseteq \cal O$, a {\em location} in $T$ is a pair $(i,a)\in I\times\A$ such that $a=c_i$ for some $c\in T$. 
Note that for any fixed $i,j\in I$, for all $c\in\cal O$ the atoms $c_i$ and $c_j$ are related in the same way in $\A_0$ (i.e.~with respect to equality and binary relations in $\sigma_0$). We say that two locations $(i,a)$ and $(j,b)$ in $T$ are in:
\begin{itemize}
\item an {\em equational conflict}, if $i\neq j$ but $a=b$, and
\item a {\em relational conflict}, if $a$ and $b$ are related in $\A_0$ but not in the same way as $c_i$ and $c_j$ for $c\in\cal O$.
\end{itemize}
(A situation where $a \neq b$ are not related by any relation in $\sigma_0$ at all does not constitute a conflict, even if $c_i$ and $c_j$ are related.) 
As $a$ and $b$ are related if they are equal, we note that an equational conflict is a special case of a relational one.

By a location in a vector $v$ we mean a location in the set $\vsup{v}$.
The prototypical examples of vectors which are free from any conflicts are cogs (or any subvectors of cogs): 
note that the locations in a cog $a^+ \between a^-$ are exactly those in $\{a^+ ,a^-\}$, and these have no conflicts if $a^+ \parallel a^-$ is a duo.

In the following proof we will often manipulate many duos and cogs at once, 
so we will benefit from a concise notation for them. 
An $\cal O$-duo will be denoted by a single letter as $a^\pm$; 
its constituent parts will then be denoted by $a^+$ and $a^-$, 
so that $a^\pm=a^+\parallel a^-$. 
Sets of duos will be denoted with capital letters such as $A^\pm$, 
and sometimes we will slightly abuse this notation and write 
$A^\pm$ to mean $\bigcup_{a^\pm\in A^\pm}\{a^+,a^-\}$,
$A^+$ to mean $\bigcup_{a^\pm\in A^\pm}\{a^+\}$, 
and $A^+ a$ to mean $\bigcup_{a^\pm\in A^\pm}\{a^+ a\}$.

\subsection{Conflict resolution lemmas}
The following sister lemmas, relying on free amalgamation as distilled in Lemma~\ref{lem:free-fresh}, show how to merge conflict-free subsets of $\cal O$ in a way that avoids introducing new conflicts. 
This will be useful in Sections~\ref{subsec:unobstructed} and~\ref{subsec:unambiguous}.

\begin{lemma}\label{lem:?}
    Let $K$ and $V_0 \subseteq V$ be finite subsets of $\cal O$ such that 
    both $V_0\cup K$ and $V$ are free from equational conflicts.
    Then there exists some $\pi\in\Aut(\A)$ that, while fixing all atoms in $S$ and in $\sqrt{V_0}$, makes $V \cup \pi(K)$ free from equational conflicts.
\end{lemma}
\begin{proof}
    Fix $V_0, V$ and induct on the number of equationally conflicting locations in $V \cup K$. 
    Take any such locations $(i,a)$ and $(j,a)$, where $i \neq j$; 
    without loss of generality $(i,a)$ is a location in $K$ and $(j,a)$ a location in $V$. 

    Suppose that $(i', a)$ were a location in $V_0$.
    Since there are no equational conflicts in $V_0\cup K$ or in $V$, 
    we see that $i = i'$ and $i' = j$, which contradicts our assumption.
    So $a \not\in \sqrt{V_0}$.
    Also $a\not\in S$, as $\cal O$ is $S$-ordered.
    Put: 
    \[
        X = S \cup \sqrt{V \cup K} \setminus \{a\},
    \]
    and note that $X$ contains $S \cup \sqrt{V_0}$.
    Use Lemma~\ref{lem:free-fresh} (putting $z=a$ and $Y=\emptyset$) to obtain an automorphism $\tau \in \Aut(\A/X)$ such that $\tau(a) \not\in X \cup \{a\}$. 
    
    In the set $V\cup\tau(K)$, the conflicting location $(i,a)$ disappears and no new equational conflicts are created, 
    so the number of equationally conflicting locations drops compared to $V\cup K$.
    Because $V_0 \cup \tau(K)$ is still equationally conflict-free,
    the inductive hypothesis gives us some $\pi \in \Aut(\A / S \cup \sqrt{V_0})$ such that $V \cup \pi \tau(K)$ is free from equational conflicts.
\end{proof}

\begin{lemma}\label{lem:!}
    Let $K$ and $V_0 \subseteq V$ be finite subsets of $\cal O$ such that both $V_0\cup K$ and $V$ are free from relational conflicts.%
    \footnote{Here and in Lemma~\ref{lem:?}, we may weaken the assumption that $V_0 \subseteq V$ and only require every location in $V_0$ to be a location in $V$.} 
    Then there exists some $\pi\in\Aut(\A)$ that, while fixing all atoms in $S$ and in $\sqrt{V_0}$, makes $V \cup \pi(K)$ free from relational conflicts.
\end{lemma}
\begin{proof}
    By Lemma~\ref{lem:?} we may assume that $V \cup K$ is free from {\em equational} conflicts. 
    As before, fix $V_0, V$ and proceed by induction on the number of relationally conflicting locations in $V \cup K$.
    
    Let $(i, a)$ and $(j, b)$ be in a relational conflict; 
    without loss of generality $(i, a)$ is a location in $K$ and $(j, b)$ in $V$. 
    This is not an equational conflict, so $a\neq b$ (but possibly $i=j$). 
    Since there are no conflicts in $V_0\cup K$ or in $V$, 
    we see that $(i, a)$ is not a location in $V$ and $(j, b)$ is not a location in $V_0 \cup K$~--~i.e.~that, 
    since there are no equational conflicts in $V \cup K$, we have $a\not\in \sqrt{V}$ and $b\not\in\sqrt{V_0\cup K}$.
    Also, $a,b\not\in S$.
    
    Let $Y$ consist of all the atoms $b'$ that are in a relational conflict with $(i,a)$ in $V\cup K$; 
    we have just shown that $Y$ does not contain $a$ and is disjoint with $S\cup\sqrt{V_0\cup K}$. 
    Put:
    \[
        X = S \cup \sqrt{V \cup K} \setminus (Y \cup \{a\}).
    \]
    Then $X, Y, \{a\}$ are pairwise disjoint, and $X$ contains $S \cup \sqrt{V_0}$. 
    It also contains all atoms in $\sqrt{K}$ except $a$.
    Using Lemma~\ref{lem:free-fresh}, find some $\tau\in \Aut(\A/X)$ such that $\tau(a) \not\in X \cup Y \cup \{a\}$ 
    and $\tau(a)$ is not related to any atom in $Y$. 
    In $V\cup\tau(K)$ the conflicting location $(i,a)$ disappears and no new conflicts are created, 
    so the conclusion follows from the inductive hypothesis.
\end{proof}

\subsection{Conflict-free vectors}\label{subsec:unobstructed}
\begin{claim}\label{claim:!-free-decomposition}
    If $v \in \Ker_\EE \cal O$ is free from conflicts, then it can be written as a sum of $\cal O$-cogs:
    \[
        v = \sum_{a^\pm \in A^\pm} \lambda_{a^\pm} \cdot a^+ \between a^-
    \]
    with $\lambda_{a^\pm} \in \EE$, where moreover $\vsup{v} \cup A^\pm$ is free from conflicts.
\end{claim}
\begin{claimproof}
We proceed by induction on the dimension $|I|$, 
noting that when $I = \emptyset$ we just have $v = v() \cdot () = v() \cdot ( \between )$.

So suppose $I$ is non-empty; let $j \in I$ be the greatest element. 
Group the terms in $v$ by their greatest atom so that $v = v^1 + v^2 + \cdots + v^k$.
We now induct on $k$.
If $k \leq 1$, we are done: as $v|^{-j} = 0$ we must have $v = 0$ (and $k = 0$), so the empty sum will do.
Otherwise \[
    v = v^{j:a} + v^{j:b} + v'
\]
for some $a\neq b\in \A$. By Lemma~\ref{lem:balanced-projected-subvector}, $v^{j:a}|^{-j}$ is balanced, and it is conflict-free, as every location in it is also a location in $v$.
By the outer inductive hypothesis, we get \[
    v^{j:a} = (v^{j:a}|^{-j}) a = \sum_{a^\pm\in A^\pm} (\lambda_{a^\pm} \cdot a^+ \between a^-)a
\]
where $\vsup{v^{j:a}|^{-j}}\cup A^\pm$ is free from conflicts, which immediately implies that $\vsup{v^{j:a}}\cup A^\pm a$ is free from conflicts as well. 
Note that if $\pi \in \Aut(\A/S)$ fixes every atom in $\sqrt{\vsup{v^{j:a}}}$ (which includes $a$), then
\[
    v^{j:a} = \pi(v^{j:a}) = \sum_{a^\pm \in A^\pm} \lambda_{a^\pm} \cdot \pi a^+ \between \pi a^-,
\]
so by Lemma~\ref{lem:!} (putting $K=A^\pm a$, $V_0=\vsup{v^{j:a}}$, and $V=\vsup{v}$) we may assume without loss of generality that
$\vsup{v}\cup A^\pm a$ is free from conflicts.

Similarly, we can write \[
    v^{j:b} = \sum_{b^\pm\in B^\pm} (\mu_{b^\pm} \cdot b^+ \between b^-)b
\]
and apply Lemma~\ref{lem:!} again (putting $K=B^\pm b$, $V_0=\vsup{v^{j:b}}$, and $V=\vsup{v}\cup A^\pm a$) to assume that
\[
    \vsup{v} \cup A^\pm a \cup B^\pm b
\]
is free from conflicts.

We now invent a new element $z$, on which we impose the following relations with $S \cup \sqrt{A^\pm a \cup B^\pm b} \subseteq \A$: 
\begin{enumerate}
    \item $a, b < z$ (noting that $\sqrt{A^\pm} < a, \sqrt{B^\pm} < b$) and, for each $s \in S$, $z < s \;:\Longleftrightarrow\; a, b < s$;
    
    \item for each unary relation $P \in \sigma_0$:
    \[
        P(z) \;:\Longleftrightarrow\; P(a) \iff P(b);
    \]
    \item for each binary relation $R \in \sigma_0$ and $s \in S$, $a^\pm \in A^\pm$, $b^\pm \in B^\pm$, $i \in I \setminus \{j\}$:
    \begin{itemize}
        \item $R(z, s) \;:\Longleftrightarrow\; R(a, s) \iff R(b, s)$,
        \item $R(z, a^+_i) \;:\Longleftrightarrow\; R(a, a^+_i)$,
        \item $R(z, b^+_i) \;:\Longleftrightarrow\; R(b, b^+_i)$,
        \item $R(z, a^-_i) \;:\Longleftrightarrow\; R(a, a^-_i)$,
        \item $R(z, b^-_i) \;:\Longleftrightarrow\; R(b, b^-_i)$,
        \item $R(z, a)$ and $R(z,b)$ are both false,
        \item and symmetrically for $R(-, z)$.
    \end{itemize}
    These are consistent as there are no equational conflicts.
    (For instance, if $a^+_i = b^-_{i'}$ then $i = i'$, and $R(a, a^+_i) \iff R(b, b^-_i)$ holds since $a^+ a$ and $b^- b$ are both in $\cal O$.)
\end{enumerate}
To see that the $\sigma$-structure $S \cup \sqrt{A^\pm a \cup B^\pm b} \cup \{z\}$ still embeds into $\A$, 
suppose towards a contradiction that it contains a forbidden $\sigma_0$-substructure $F$.
Then $F$ must contain $z$.
Since any two elements in $F$ are necessarily related, we must have $a \not\in F$ and $b \not\in F$.
Similarly, whenever $F$ contains an atom $x_i$ for any $x\in A^\pm\cup B^\pm$, it does not contain a distinct atom $y_i$ for any $y\in A^\pm\cup B^\pm$.
It follows that, fixing any $a^\pm\in A^\pm$,
\[
    s \mapsto s,\quad x_i \mapsto a^+_i,\quad z \mapsto a
\]
defines an injective function $\phi : F \to \A_0$,
which is furthermore an embedding (we only need to check this for pairs!) because $A^\pm \cup B^\pm$ is conflict-free and any $x_i, y_{i'}$ for $i \neq i'$ are related.
This is a contradiction. We may therefore assume that $z \in \A$.

It is now routine to check that $a^+ a \parallel a^- z$ and $b^+ b \parallel b^- z$ are $\cal O$-duos for all $a^\pm \in A^\pm, b^\pm \in B^\pm$,
and that 
\[
    A^+ a \cup A^- z \cup B^+ b \cup B^- z
\]
is free from conflicts.
From Lemma~\ref{lem:!} (putting $K = A^- a \cup B^- z$, $V_0 = A^\pm a \cup B^\pm b$, $V = \vsup{v} \cup A^\pm a \cup B^\pm b$) we may assume that 
\[
\vsup{v} \cup A^+ a \cup A^- z \cup B^+ b \cup B^- z
\] 
is also free from conflicts.
(Alternatively, we could have explicitly ensured this when defining $z$.)
Then the vector:
\begin{align*}
    v'' &= v
    - \sum_{a^\pm\in A^\pm} \lambda_{a^\pm} \cdot a^+ a \between a^- z 
    - \sum_{b^\pm\in B^\pm} \mu_{b^\pm} \cdot b^+ b \between b^- z  \\
    &= v^{j:a}|^{-j} z + v^{j:b}|^{-j} z + v',
\end{align*}
when grouped into subvectors by the largest atom in each term, has at least one fewer component than $v$.
By the inner inductive hypothesis, we may write
\[
    v'' = \sum_{c^\pm\in C^\pm} \kappa_{c^\pm} \cdot c^+ \between c^-
\]
with $\vsup{v''} \cup C^\pm$ conflict-free, and one
last application of Lemma~\ref{lem:!} (putting $K = C^\pm$, $V_0 = \vsup{v''}$, and $V = \vsup{v} \cup A^+ a \cup A^- z \cup B^+ b \cup B^- z$) allows us to conclude that
\[
    \vsup{v} \cup A^+ a \cup A^- z \cup B^+ b \cup B^- z \cup C^\pm
\]
is conflict-free as well.
We conclude that
\begin{align*}
    v = 
      \sum_{a^\pm\in A^\pm} \lambda_{a^\pm} \cdot a^+ a \between a^- z 
    + \sum_{b^\pm\in B^\pm} \mu_{b^\pm} \cdot b^+ b \between b^- z 
    + \sum_{c^\pm\in C^\pm} \kappa_{c^\pm} \cdot c^+ \between c^-
\end{align*}
is a decomposition of $v$ into a sum of $\cal O$-cogs as required.
\end{claimproof}

\subsection{Vectors without equational conflicts}\label{subsec:unambiguous}
\begin{claim}\label{claim:?-free-decomposition}
     If $v \in \Ker_\EE \cal O$ has no equational conflicts, then it can be written as a sum of $\cal O$-cogs:
    \[
        v = \sum_{a^\pm \in A^\pm} \lambda_{a^\pm} \cdot a^+ \between a^-
    \]
    with $\lambda_{a^\pm} \in \EE$, where moreover $\vsup{v} \cup A^\pm$ has no equational conflicts.
\end{claim}
\begin{claimproof}
We proceed again by induction, first on $|I|$ then on the number of relational conflicts in $v$.
The outer base case $I = \emptyset$ is trivial~--~we have $v = v() \cdot ( \between )$, and no conflicts arise~--~and the inner base case is just Claim~\ref{claim:!-free-decomposition}.

Suppose that a location $(i, a)$ is part of a relational conflict in $v$.
Since every location in $v^{i:a}|^{-i}$ is also a location in $v$, we know that $v^{i:a}|^{-i}$ has no equational conflicts, and by Lemma~\ref{lem:balanced-projected-subvector} it is balanced.
By the outer inductive hypothesis, we get:
\[
    v^{i:a} = (v^{i:a}|^{-i}) a = \sum_{a^\pm\in A^\pm} (\lambda_{a^\pm} \cdot a^+ \between a^-) a
\]
where $\vsup{v^{i:a}|^{-i}}\cup A^\pm$ has no equational conflicts, 
which immediately implies that $\vsup{v^{i:a}}\cup A^\pm a$ is free from equational conflicts as well. 
By Lemma~\ref{lem:?} (putting $K=A^\pm a$, $V_0=\vsup{v^{i:a}}$, and $V=\vsup{v}$) 
we may assume that $\vsup{v}\cup A^\pm a$ has no equational conflicts.

Take any location $(j,b)$ that is in a relational conflict with $(i,a)$ in $v$. 
Then $b\not\in\sqrt{A^\pm a}$. 
To see this, note that $b=a$ would imply $j=i$ (since $v$ has no equational conflicts), which is not a conflict. 
On the other hand, if $b = a^+_k$ for some $a^+ \in A^\pm$ and $k \in I \setminus \{i\}$ (the case of $a^-$ is identical) 
then $j=k$, but this is not a conflict either, since $a^+a\in\cal O$ and $a=(a^+a)_i$.
Also, $b \not\in S$.

Let $Y$ consist of all the atoms $b$ that are in a relational conflict with $(i,a)$ in $v$. 
As $a \not\in S$ and $a \not\in \sqrt{A^\pm}$, we have shown that
\[
    X = S \cup \sqrt{\vsup{v} \cup A^\pm} \setminus (Y \cup \{a\})   
\]
contains $S \cup \sqrt{A^\pm}$ and that $X, Y, \{a\}$ are pairwise disjoint.
Use Lemma~\ref{lem:free-fresh} (putting $z=a$) to find $\tau \in \Aut(\A/X)$ such that $\tau(a) \not\in X \cup Y \cup \{a\}$, is greater than $a$, and is not related to any of $Y \cup \{a\}$. 
Write $a'=\tau(a)$.
Then, for any $a^\pm \in A^\pm$, it follows by Lemma~\ref{lem:cog-fresh-single} that $a^+ a \parallel a^- a'$ is an $\cal O$-duo. 
Moreover, $\vsup{v} \cup A^+ a \cup A^- a'$ has no equational conflicts, and the vector
\begin{align*}
    v' = v - \sum_{a^\pm\in A^\pm} \lambda_{a^\pm} \cdot a^+ a \between a^- a'
    &= v - v^{i:a} + (v^{i:a}|^{-i}) a' 
\end{align*}
has strictly fewer relationally conflicting locations than $v$, as the location $(i,a)$ disappears from it.
The inner inductive hypothesis tells us that we may write
\[
    v' = \sum_{b^\pm\in B^\pm} \mu_{b^\pm} \cdot b^+ \between b^-
\]
with $\vsup{v'} \cup B^\pm$ free from equational conflicts.

Since any location in $\vsup{v'}$ is a location in $\vsup{v} \cup A^+ a \cup A^- a'$, Lemma~\ref{lem:?} (putting $K = B^\pm$, $V_0 = \vsup{v'}$, and $V = \vsup{v} \cup A^+a \cup A^- a'$) allows us to assume that 
\[
    \vsup{v} \cup A^+ a \cup A^-a' \cup B^\pm
\]
is also free from equational conflicts. We conclude that 
\[
    v =\sum_{a^\pm\in A^\pm} \lambda_{a^\pm} \cdot a^+ a \between a^-a'+ \sum_{b^\pm\in B^\pm} \mu_{b^\pm} \cdot b^+ \between b^-
\]
as required.
\end{claimproof}

\subsection{Arbitrary vectors}
We are in a position to restate and prove Theorem~\ref{thm:cog-span-generally}:
\begin{theorem}
    Any $v \in \Ker_\EE \cal O$ can be written as
    \[
        v = \sum_{a^\pm \in A^\pm} \lambda_{a^\pm} \cdot a^+ \between a^-
    \]
    with $\lambda_{a^\pm} \in \EE$.
\end{theorem}
\begin{proof}
This is similar to the proof of Claim~\ref{claim:?-free-decomposition}, but simpler.
We proceed again by induction, first on $|I|$ then on the number of equational conflicts in $v$.
The outer base case $I = \emptyset$ is trivial as before,  
and the inner base case is Claim~\ref{claim:?-free-decomposition}.

Suppose that a location $(i, a)$ is part of an equational conflict in $v$. 
By the outer inductive hypothesis, we get:
\[
    v^{i:a}
    = (v^{i:a}|^{-i}) a 
    = \sum_{a^\pm\in A^\pm} (\lambda_{a^\pm} \cdot a^+ \between a^-) a.
\]
Then neither $S$ nor $\sqrt{A^\pm}$ contains $a$,
so 
\[
    X = S \cup \sqrt{\vsup{v} \cup A^\pm} \setminus \{a\}
\]
contains $S \cup \sqrt{A^\pm}$.
Using Lemma~\ref{lem:free-fresh} (putting $z=a$ and $Y=\emptyset$), 
find $\pi \in \Aut(\A/X)$ such that $\pi(a)$ is not in $X$, is greater than $a$, and is otherwise unrelated to $a$. 
Write $a'=\pi(a)$. 
Then, for any $a^\pm \in A^\pm$, it follows by Lemma~\ref{lem:cog-fresh-single} that $a^+ a \parallel a^- a'$ is an $\cal O$-duo. 
Moreover, the vector
\begin{align*}
    v' = v - \sum_{a^\pm\in A^\pm} \lambda_{a^\pm} \cdot a^+ a \between a^- a'
    &= v -  v^{i:a} + (v^{i:a}|^{-i}) a' 
\end{align*}
has strictly fewer equationally conflicting locations than $v$, as the location $(i,a)$ disappears from it.
It follows from the inner inductive hypothesis that we can write
\[
    v' = \sum_{b^\pm \in B^\pm} \mu_{b^\pm} \cdot b^+ \between b^-,
\]
which gives a decomposition $v = v' + \sum_{a^\pm \in A^\pm} \lambda_{a^\pm} \cdot a^+ a \between a^- a'$ as required.
\end{proof}

%% file: src/appendix-coefficients.tex
\section{Other proofs for Section~\ref{sec:equivariant-subspaces}}\label{sec:appendix-eqsubsp}

\subsection{Proof of Theorem~\ref{thm:coeff-approximation}}
For notational simplicity, we will prove the result for a single equivariant ordered orbit $\cal O \subseteq \A^I$.
The general multi-orbit case is very similar, because we will be projecting onto a single orbit anyway.

Consider the $2^{|I|}$ projected $S$-ordered orbits $\mathcal{O}|^{J}$ for $J \subseteq I$.
Suppose that
\[
    f : \mathcal{O}|^J \to \mathcal{O}|^{J'}
\]
is an $\Aut(\A/S)$-equivariant bijection.
Take any $a \in \mathcal{O}|^J$,
and enumerate its entries as $a_1 < \dots < a_{|J|}$.
Similarly, enumerate the entries of $f(a)$ as $b_1 < \dots < b_{|J'|}$.
Then $\{a_1, \dots, a_{|J|}\} = \{ b_1, \dots, b_{|J'|} \}$ because $\A$ has no algebraicity \cite[Thm.~7.1.8]{hodges1993model}; 
since the orbits are ordered, we must have $|J| = |J'|$ and $a_1 = b_1, \dots, a_{|J|} = b_{|J'|}$.
That is, $f$ must be the obvious function that reindexes a $J$-tuple to a $J'$-tuple~--~hence we will write $a^{/J'}$ instead of $f(a)$, leaving $f$ implicit.

Now, let $\mathcal{Q}_{1} = \mathcal{O}|^{J_1}, \dots, \mathcal{Q}_{t} = \mathcal{O}|^{J_t}$ be distinct projected $S$-ordered orbits up to $\Aut(\A/S)$-equivariant bijections,
enumerated in such a way that $|J_1| \geq |J_2| \geq \dots \geq |J_t|$. (In particular, $J_1=I$ and $J_t=\emptyset$.)

\begin{definition}
    For $i = 1, \dots, t$, let $\Jclass_i$ consist of all sets $J$ such that $\mathcal{O}|^J$ is in an $\Aut(\A/S)$-equivariant bijection with $\mathcal{Q}_i$.
    Assemble all $|\Jclass_i|$ projections into a single map
    \[
        (-){\restriction_i} : \Lin_\FF\mathcal{O} \to \Lin_{\FF^{\Jclass_i}} \mathcal{Q}_i.
    \]
    To be more precise $v{\restriction_i}(a)$ is, for $a\in{\cal Q}_i$, the $\Jclass_i$-tuple whose $J$-th entry is $v|^J(a^{/J}) \in \FF$.
\end{definition}
    It is straightforward to check that each $(-){\restriction_i}$ is $\Aut(\A/S)$-equivariant and linear. 
    
    With these maps defined, let us recall the spaces described in Theorem~\ref{thm:coeff-approximation}:
\begin{definition}
    Given an equivariant subspace $W \subseteq \Lin_\field \mathcal{O}$, 
    we write
    \[
        \widehat W 
        = \{ v \in \Lin_\field \mathcal{O} 
        \mid\ \forall i : 
            \underbrace{
                \forall a \in \mathcal{Q}_i, 
                \exists a' \in \mathcal{Q}_i, 
                \exists w \in W \text{ such that } 
                    v{\restriction_i}(a) = w{\restriction_i}(a')
            }_{\text{i.e.~$v{\restriction_i}(\mathcal{Q}_i) \subseteq W{\restriction_i}(\mathcal{Q}_i)$}} 
        \ \}.
    \]
\end{definition}
    Then $\widehat{W}$ is another equivariant subspace of $\Lin_\field \mathcal{O}$, 
    which can be seen as an approximation of $W$ by only looking at coefficients under $\restriction_1, \dots, \restriction_t$;
    in particular $W \subseteq \widehat W$.
    But it turns out that $\widehat{W}$ coincides with $W$:
\begin{lemma}\label{lem:coeff-approximation}
    For $i = 1, \dots, t, t+1$,
    \begin{align*}
        &\widehat W \cap \ker(\restriction_i) \cap \ker(\restriction_{i+1}) \cap \cdots \cap \ker(\restriction_t) \\
        \subseteq{} &W \cap \ker(\restriction_i) \cap \ker(\restriction_{i+1}) \cap \cdots \cap \ker(\restriction_t).
    \end{align*}
    This gives Theorem~\ref{thm:coeff-approximation} when $i = t+1$.
\end{lemma}

\begin{proof}
We proceed by induction on $i$, working with fewer and fewer zeroness assumptions as we go.
The base case $i = 1$ is trivial, since $\ker(\restriction_1) = \{0\}$. 
Indeed, $\Jclass_1=\{I\}$, so $v{\restriction_1}$ is the identity map.

To prove the containment for any given $i>1$, 
we allow $v {\restriction_i}$ to be non-zero.
But $v {\restriction_i}$ satisfies the next best property:
\begin{claim}
    The image of \[
        \widehat W \cap \ker(\restriction_{i+1}) \cap \cdots \cap \ker(\restriction_t)
    \] 
    under $\restriction_{i}$ is contained in $\Ker_{W {\restriction}_{i} (\mathcal{Q}_{i})} \mathcal{Q}_{i}$. 
\end{claim}
\begin{claimproof}
    Take any $v \in \widehat W$ such that  $v {\restriction_{i'}} = 0$ for all $i'>i$.
    That $v {\restriction_{i}} \in \Lin_{W{\restriction_{i}}(\mathcal{Q}_{i})} \mathcal{Q}_{i}$ is clear from the definition of $\widehat W$.
    Recall that $\mathcal{Q}_{i}$ = $\mathcal{O}|^{J_{i}}$.
    For every $j \in J_{i}$,
    we need to prove that $v {\restriction_{i}} |^{-j} = 0$. 
    Unravelling the definitions, given any $a\in\mathcal{O}|^{J_{i}\setminus\{j\}}$, we need the $J$-th entry of $v {\restriction_{i}} |^{-j}(a)$ to be $0$, for every $J\in \Jclass_i$.
    
    So take any $J \in \Jclass_{i}$.
    The unique ordered bijection between $J_{i}$ and $J$ restricts to one between $J_{i} \setminus \{j\}$ and $J \setminus \{j'\}$, for some $j'\in J$. 
    Write $J'=J \setminus \{j'\}$.
    Then $J'$ belongs to some $\Jclass_{i'}$ with $i' > i $,
    so $v {\restriction_{i'}} = 0$. Now calculate (with $a\in\mathcal{O}|^{J_{i}\setminus\{j\}}$, $b\in{\cal Q}_i$, $c\in {\cal O}|^J$ and $d\in\cal O$):
    \newtagform{noparen}{}{}
    \usetagform{noparen}    
\begin{align*}
	(v{\restriction_i}|^{-j}(a))_J &= \sum_{b|^{-j}=a}(v{\restriction_i}(b))_J
	= \sum_{b|^{-j}=a}v|^J(b^{/J}) 
	= \sum_{c|^{-j'}=a^{/J'}}v|^J(c) \\
	&= \sum_{d|^{J'}=a^{/J'}}v(d)
	\,\, = v|^{J'}(a^{/J'}) = (v{\restriction_{i'}}(a))_{J'} = 0.
    \tag{\claimqedhere} 
\end{align*}    
\end{claimproof}

In light of Remark~\ref{rem:FF-EE}, from Theorem~\ref{thm:cog-span-generally} we get
$\Ker_{W {\restriction_{i}} (\mathcal{Q}_{i})} \mathcal{Q}_{i} \subseteq \Cog_{W {\restriction_{i}} (\mathcal{Q}_{i})} \mathcal{Q}_{i}$.
Furthermore:
\begin{claim}\label{claim:cogs-arise-everywhere}
    $\Cog_{W {\restriction}_{i} (\mathcal{Q}_{i})} \mathcal{Q}_{i}$ is contained in
    the image of \[
        W \cap \ker(\restriction_{i+1}) \cap \cdots \cap \ker(\restriction_t)
    \] under $\restriction_{i}$. 
\end{claim}
\begin{claimproof}
    Consider any ${\cal Q}_i$-cog with a coefficient of $\lambda\in W {\restriction_{i}} (\mathcal{Q}_{i})$, and let $w\in W$ and $a\in{\cal Q}_i$ be such that $\lambda = w{\restriction_{i}}(a)$.
    Let $S'$ consist of $S$ together with every atom appearing in $w$ but not in $a$.
    We generalise the construction used in the proof of Theorem~\ref{thm:cogs-arise-everywhere}.
    
    Apply Lemma~\ref{lem:cog-fresh-full} and Remark~\ref{rem:duo} to get automorphisms $\pi_j$ for $j \in J_i$ 
    such that $a \parallel \prod_{j \in J_i} \pi_j a$ is an $\mathcal{Q}_i$-duo, 
    where each $\pi_j$ fixes $S'$ and all $a_{j'}, \pi_{j'}(a_{j'})$ for $j'\neq j$. 
    Since all $\mathcal{Q}_i$-duos are in the same orbit, it is enough to show that 
    the cog corresponding to this particular duo, with the coefficient $\lambda$, 
    belongs to the $\restriction_i$-image as in the statement of the claim.
    Put: \[
        w' = \prod_{j \in J_{i}} (\mathrm{id} - \pi_j) w \in W.
    \]
    For any $1 \leq i' \leq t$, noting that no more atoms can appear in $w{\restriction_{i'}}$ than in $w$, we have
    \begin{align*}
        w' {\restriction_{i'}}
        = \prod_{j \in J_{i}} (\mathrm{id} - \pi_j) w {\restriction_{i'}} 
        = \sum_{c \in C_{i'}} \sum_{J' \subseteq J_i} (-1)^{|J'|}  
            w {\restriction_{i'}}(c) \cdot \left(\prod_{j \in J'} \pi_j c\right)
    \end{align*}
    where \[
        C_{i'} = \{ c\in{\cal Q}_{i'} \mid\ 
            \{c_j : j\in J_{i'}\} \supseteq \{a_j : j\in J_i\} 
        \ \}.
    \]
    (The formula used in the proof of Theorem~\ref{thm:cogs-arise-everywhere} is a special case of this for $i'=1$ so that $J_{i'}=I$ and ${\cal Q}_{i'}={\cal O}$.) 
    Now, if $i'>i$ then $C_{i'}$ is empty, and so $w' {\restriction_{i+1}} = \cdots = w' {\restriction_{t}} = 0$. Moreover, if $i'=i$ then $C_i = \{a\}$ and we obtain the cog from before:
    \[
        w' {\restriction_{i}} = \lambda \cdot \left(a \between \prod_{j \in J_{i}} \pi_j a\right).
    \]
So $w'$ is a witness for the inclusion from the claim.
\end{claimproof}

This is enough to establish Lemma~\ref{lem:coeff-approximation} for $i+1$, assuming it holds for $i$.
Indeed, given $v \in \widehat W \cap \ker(\restriction_{i+1}) \cap \dots \cap \ker(\restriction_t)$, by the preceding claims
we can find \[
    w \in W \cap \ker(\restriction_{i+1}) \cap \dots \cap \ker(\restriction_t)
\] such that 
$v {\restriction_i} = w {\restriction_i}$.
But then $(v - w) {\restriction_{i}} = 0$, so $v - w$ lies in $\ker(\restriction_{i})$, $\ker(\restriction_{i+1}), \dots, \ker(\restriction_t)$ as well as $\widehat{W}$ (which contains $W$).
It follows from the inductive hypothesis that \[
    v - w \in W \cap \ker(\restriction_i) \cap \ker(\restriction_{i+1}) \cap \dots \cap \ker(\restriction_t),
\]
so $v = (v - w) + w$ is in $W \cap \ker(\restriction_{i+1}) \cap \dots \cap \ker(\restriction_t)$ as well.
This completes the proof of Theorem~\ref{thm:coeff-approximation}.
\end{proof}

\subsection{Proof of Corollary~\ref{cor:2d-length}}
Now we specialise (as already done above) to the case of a single equivariant ordered orbit $\cal O \subseteq \A^d$.
The maps $\{\restriction_i : \Lin_\FF \cal O \to \Lin_{\FF^{\Jclass_i}} \cal Q_i\}_{1 \leq i \leq t}$ are as before.
We first bound the length of $\Lin_\field(\cal O)$ from above:

\begin{corollary}\label{cor:length-upper-bound}
    Let $W_0 \subset W_1 \subset \cdots \subset W_l$ be a chain of $\Aut(\A/S)$-equivariant subspaces of $\Lin_\FF(\mathcal{O})$.
    Then $l \leq 2^{|I|}$. 
\end{corollary}
\begin{proof}
    We obtain $t$ sequences of finite-dimensional vector spaces:
    \begin{align*}
        W_0 {\restriction_1} (\mathcal{Q}_1) \subseteq W_1 {\restriction_1} (\mathcal{Q}_1) \subseteq \cdots \subseteq W_l {\restriction_1} (\mathcal{Q}_1) \subseteq \FF^{\Jclass_1}, \\
        W_0 {\restriction_2} (\mathcal{Q}_2) \subseteq W_1 {\restriction_2} (\mathcal{Q}_2) \subseteq \cdots \subseteq W_l {\restriction_2} (\mathcal{Q}_2) \subseteq \FF^{\Jclass_2}, \\
        \vdots\phantom{,}\\
        W_0 {\restriction_t} (\mathcal{Q}_t) \subseteq W_1 {\restriction_t} (\mathcal{Q}_t) \subseteq \cdots \subseteq W_l {\restriction_t} (\mathcal{Q}_t) \subseteq \FF^{\Jclass_t}.
    \end{align*}
    At each of the $l$ steps, at least one of the $t$ containments must be strict:
    if two equivariant subspaces $W, W' \subseteq \Lin_\FF \cal O$ satisfy $W {\restriction_i}(\mathcal{Q}_i) = W' {\restriction_i}(\mathcal{Q}_i)$ for all $i=1,\ldots, t$,
    then $W = \widehat W = \widehat{W'} = W'$ by Lemma~\ref{lem:coeff-approximation}.
    Now the $i$-th sequence can accommodate at most $\dim(\field^{\Jclass_i}) = |\Jclass_i|$ strict containments,
    so we must have $l \leq |\Jclass_1| + |\Jclass_2| + \dots + |\Jclass_t| = 2^{|I|}$. 
\end{proof}

To complement the upper bound from Corollary~\ref{cor:length-upper-bound}, we now exhibit a chain of equivariant subspaces whose length is precisely $\sum_{i=1}^t |\Jclass_i|$, using the same idea as~\cite[Cor.~4.12]{BFKM24}.

Pick some $a\in \cal O$, and let $\pi_j$ (for $j \in I$) be the automorphisms from Lemma~\ref{lem:cog-fresh-full}.
Given $J \in \Jclass_i$, define a vector
\begin{align*}
    v_{J} =  \prod_{j \in J} (1 - \pi_j) a
\end{align*}
in $\Lin_\FF(\cal O)$. 
Write $\langle v_J \rangle$ for the linear span of the orbit of $v_J$; this is an equivariant subspace of $\Lin_\FF(\cal O)$.
Given $J' \in \Jclass_{i'}$, we compute:
\[
    (\langle v_{J} \rangle {\restriction_{i'}} (\mathcal{Q}_{i'}))_{J'} =
    \begin{cases}
        \{0\} & \text{if $J' \not\supseteq J$}, \\
        \FF & \text{otherwise}.
    \end{cases}
\]
Enumerating each $\Jclass_i$ in any order as $J_i^1, J_i^2, \dots, J_i^{|\Jclass_i|}$,
we thus obtain a chain
\begin{align*}
    \langle  \rangle
    \subset{} &\langle v_{J_1^1} \rangle 
    \subset{} \langle v_{J_1^1}, v_{J_1^2} \rangle  
    \subset{} \cdots 
    \subset{} \langle v_{J_1^1}, v_{J_1^2}, \dots, v_{J_1^{|\Jclass_1|}} \rangle \\
    \subset{} &\langle v_{J_1^1}, v_{J_1^2}, \dots, v_{J_1^{|\Jclass_1|}}, v_{J_{2}^1} \rangle 
     \subset{} \cdots
\end{align*}
of length $|\Jclass_1| + |\Jclass_2| + \cdots + |\Jclass_t| = 2^{|I|}$~--~observe that 
$\langle x , y \rangle {\restriction_{i'}} (\mathcal{Q}_{i'}) = \langle x \rangle {\restriction_{i'}} (\mathcal{Q}_{i'}) + \langle y \rangle {\restriction_{i'}} (\mathcal{Q}_{i'})$.
We conclude that the upper bound from Corollary~\ref{cor:length-upper-bound} is tight.